\documentclass[12pt,reqno]{amsart}

\usepackage{mathrsfs}
\usepackage{amsmath}
\usepackage{amssymb}
\usepackage{amsfonts}
\usepackage{amsthm}
\usepackage{a4wide}
\usepackage{graphicx}
\usepackage{color}
\usepackage[sans]{dsfont}
\usepackage{linearA}
\usepackage{longtable}
\usepackage{pdflscape}
\usepackage{float}
\usepackage{cancel}
\usepackage{dsfont}
\usepackage{setspace}
\usepackage{array}
\usepackage{hyperref}
\hypersetup{colorlinks=true,allcolors=[rgb]{0,0,0.6}}
\usepackage[english]{babel}
\usepackage[applemac]{inputenc}
\usepackage[T1]{fontenc}
\usepackage{tikz}
\usepackage{enumitem}
\usepackage[square,numbers,sort&compress]{natbib}
\usepackage[left=2.5cm,right=2.5cm,top=3cm,bottom=3cm]{geometry}
\usepackage{soul}

\newcommand{\C}{\mathbb{C}}
\newcommand{\R}{\mathbb{R}}
\newcommand{\F}{\mathfrak{F}_s}
\newcommand{\FF}{\mathcal{F}}
\newcommand{\Real}{\mathfrak{Re}}

\newcommand{\I}{\mathbf{I}}

\newcommand{\ZZ}{\mathcal{Z}}
\newcommand{\LL}{\mathcal{L}}

\newcommand{\E}{\mathcal{E}}

\newcommand{\abs}[1]{\left|#1\right|}
\newcommand{\norm}[1]{\left\|#1\right\|}
\newcommand{\bra}[1]{\left\langle #1 \right|}
\newcommand{\ket}[1]{\left| #1 \right\rangle}

\newcommand{\ps}[2]{\left\langle #1,#2 \right\rangle}

\newcommand{\diff}{\mathop{}\!\mathrm{d}}

\newtheorem{theorem}{Theorem}[section]
\newtheorem{lemma}[theorem]{Lemma}
\newtheorem{proposition}[theorem]{Proposition}

\newtheorem{remark}[theorem]{Remark}
\newtheorem{assumption}{Hypothesis}

\numberwithin{equation}{section}

\author[S. Breteaux]{S{\'e}bastien Breteaux}
\address[S. Breteaux]{Universit{\'e} de Lorraine, CNRS, IECL, F-57000 Metz, France}
\email{sebastien.breteaux@univ-lorraine.fr}

\author[J. Faupin]{J{\'e}r{\'e}my Faupin}
\address[J. Faupin]{Universit{\'e} de Lorraine, CNRS, IECL, F-57000 Metz, France}
\email{jeremy.faupin@univ-lorraine.fr}

\author[J. Payet]{Jimmy Payet}
\address[J. Payet]{Universit{\'e} de Lorraine, CNRS, IECL, F-57000 Metz, France}
\email{jimmy.payet@univ-lorraine.fr}

\theoremstyle{plain}

\theoremstyle{plain}
\theoremstyle{plain}

\theoremstyle{plain}

\theoremstyle{remark}

\makeatother

\usepackage{babel}

\providecommand{\lemmaname}{Lemma}
\providecommand{\notationname}{Notation}
\providecommand{\propositionname}{Proposition}
\providecommand{\theoremname}{Theorem}

\begin{document}
\title[Quasi-classical ground states]{Quasi-classical ground states. I. Linearly coupled Pauli-Fierz Hamiltonians}
\begin{abstract}
We consider a spinless, non-relativistic particle bound by an external potential and linearly coupled to a quantized radiation field. The energy $\mathcal{E}(u,f)$ of product states of the form $u\otimes \Psi_f$, where $u$ is a normalized state for the particle and $\Psi_f$ is a coherent state in Fock space for the field, gives the energy of a Klein--Gordon--Schr\"odinger system. We minimize the functional $\mathcal{E}(u,f)$ on its natural energy space. We prove the existence and uniqueness of a ground state under general conditions on the coupling function. In particular, neither an ultraviolet cutoff nor an infrared cutoff is imposed. Our results establish the convergence in the ultraviolet limit of both the ground state and ground state energy of the Klein--Gordon--Schr\"odinger energy functional, and provide the second-order asymptotic expansion of the ground state energy at small coupling.
\end{abstract}
\maketitle
\tableofcontents
%\`A FAIRE
%\begin{itemize}
%\item supprimer la remarque sur les etats quasilibres? 
%\item Citer AmmariNier08
%\item Il y a des references non citees, il faudrait voir ou il faut les citer dans le texte. Supprimer les r\'ef\'erences qui ne concernent que Maxwell-Schr\"odinger. V\'erifier que l'on n'oublie rien dans la bibliographie.
%\item Ajouter une r\'ef\'erence pour l'article compagnon (cit\'e \`a la fin de l'intro).
%\item Seb: Nous n'avons pas defini les espaces de Lorentz $L^{\infty,q}$, donc j'ai restreint les hypotheses la ou on avait $p\leq\infty$ pour des espaces de Lorentz. J'ai pris une definition des espaces de Lorentz qui marche dans les deux cas: $q=\infty$ et $q<\infty$. J'ai mis une version de Holder qui regroupe les deux qu'on avait avant. Cf Grafakos.
%\end{itemize}
%\rule[0.5ex]{1\columnwidth}{1pt}
%
%
\section{Introduction}
We consider in this paper a non-relativistic, spinless quantum particle -- say, an electron -- in an external potential and coupled to a quantized, scalar radiation field. The Hilbert space for the total system is given by
\begin{equation*}
\mathcal{H} := \mathcal{H}_{\mathrm{el}}\otimes\mathcal{H}_{\mathrm{f}},
\end{equation*}
where  $\mathcal{H}_{\mathrm{el}}=L^2(\mathbb{R}^3)$ is the Hilbert space for the electron and $\mathcal{H}_{\mathrm{f}}$ is the Hilbert space for the field, given as the symmetric Fock space over the one-particle Hilbert space $\mathfrak{h}=L^2(\mathbb{R}^3)$. The full Hamiltonian is a self-adjoint operator acting on $\mathcal{H}$, of the form
\begin{equation}\label{eq:def_Hamilt}
\mathbb{H} := H_{V}\otimes\mathbf{I}_{\mathrm{f}}+\mathbf{I}_{\mathrm{el}}\otimes \mathbb{H}_{\mathrm{f}}+ \mathbb{H}_{\mathrm{int}},
\end{equation}
where $H_{V}=-\Delta+V$ is the Hamiltonian for the non-relativistic particle in the external potential $V$, $\mathbb{H}_{\mathrm{f}}$ is the Hamiltonian for the free field, $\mathbf{I}_\sharp$ stands for the identity on $\mathcal{H}_\sharp$ and $\mathbb{H}_{\mathrm{int}}$ is the interaction Hamiltonian, acting on $\mathcal{H}$. Such operators are usually called \emph{Pauli-Fierz Hamiltonians} \cite{PauliFierz38} in the literature. Their spectral theory has been thoroughly studied since the end of the nineties, including in particular ground state properties (see, among others, \cite{BachFrohlichPizzo09,BachFrohlichSigal98,BachFrohlichSigal99,Barbarouxetal10,DerezinskiGerard99,Hainzl03,GriesemerLiebLoss01,HaslerHerbst11,Sigal09, Spohn04, Gerard00, GriesemerHasler09, BachHach22, LiebLoss99} and references therein).

We focus in this paper on the case of an electron \emph{linearly coupled} to a scalar field. We consider an abstract class of linearly coupled Pauli-Fierz Hamiltonian that includes the \emph{Nelson model} \cite{Nelson64} and the \emph{Fr\"ohlich polaron model} \cite{Frohlich39}. 

We aim at studying the energy of product states
\begin{equation}\label{eq:quasi-class-en}
\mathcal{E}(u,f) := \big\langle (u\otimes \Psi_f),\mathbb{H}(u\otimes \Psi_f)\big\rangle, \quad \|u\|_{\mathcal{H}_{\mathrm{el}}}=1,\quad \|\Psi_f\|_{\mathcal{H}_{\mathrm{f}}}=1,
\end{equation}
assuming that the state of the quantized field, $\Psi_f$, is a \emph{coherent state} parametrized by $f\in\mathfrak{h}$. The functional $\mathcal{E}(u,f)$ is sometimes called \emph{quasi-classical energy}. Assuming indeed that the field degrees of freedom are `almost classical', in the sense that the creation and annihilation operator $a^*$, $a$ in $\mathcal{H}_{\mathrm{f}}$ are rescaled as $a^*_\varepsilon=\sqrt{\varepsilon}a^*$, $a_\varepsilon=\sqrt{\varepsilon}a$, see~\cite{AmmariNier08}, one shows, under suitable assumptions, that the ground state energy of the rescaled Pauli-Fierz Hamiltonian~$\mathbb{H}_\varepsilon$ converges, as~$\varepsilon\to0$, to the ground state energy of the quasi-classical energy functional \eqref{eq:quasi-class-en},  \cite{CorreggiFalconi18, CorreggiFalconiOlivieri20, CorreggiFalconiOlivieri19_02,GinibreNironiVelo06}. In order for $\mathbb{H}_\varepsilon$ to identify to a semi-bounded self-adjoint operator on $\mathcal{H}$, an ultraviolet cutoff is usually imposed to the interaction Hamiltonian.

As we  recall below, the quasi-classical energy \eqref{eq:quasi-class-en} coincides with the energy of a coupled \emph{Klein--Gordon--Schr\"odinger} system.  The variational and dynamical aspects of Klein--Gordon--Schr\"odinger  systems have been studied in the recent mathematical literature (see \cite{AmmariFalconi14,AmmariFalconi17,AmmariFalconiOlivieri21,CorreggiFalconi18, CorreggiFalconiOlivieri20, CorreggiFalconiOlivieri19_02,Falconi13}), as quasi-classical limits of Pauli-Fierz models. Existence of a ground state and of a dynamics associated to the non-linear energy functional \eqref{eq:quasi-class-en} constitute essential ingredients of the analysis. 

In this paper, under general conditions, we prove the existence of a ground state for the quasi-classical energy functional $\mathcal{E}(u,f)$ on its natural energy space. In particular, we consider a wide class of external potentials and we do not need to impose \emph{neither an infrared nor an ultraviolet cutoff} into the interaction term of $\mathcal{E}(u,f)$.

For small coupling, we verify that the ground state of $\mathcal{E}(u,f)$ is unique. In general, the field parameter $f_{\mathrm{gs}}$ of the ground state $(u_{\mathrm{gs}},f_{\mathrm{gs}})$ does not necessarily belong to the orignal one-particle Hilbert space $\mathfrak{h}$. For massless fields, we will see that~$f_{\mathrm{gs}}$ belongs to~$\mathfrak{h}$ if and only if an infrared regularization is imposed. On the other hand, no ultraviolet cutoff is needed: Denoting by $\Lambda$ the ultraviolet parameter associated to the ultraviolet cutoff introduced into the interaction Hamiltonian, our results show that both the ground states and ground state energies associated to \eqref{eq:quasi-class-en} converge as $\Lambda\to\infty$.

We also study the difference between the ground state energy for the microscopic model and its quasi-classical counterpart,
\begin{equation}\label{eq:diff_en}
\inf \sigma(\mathbb{H}) - \inf_{(u,f)}\mathcal{E}(u,f),
\end{equation}
where $\sigma(\mathbb{H})$ stands for the spectrum of the Pauli-Fierz Hamiltonian $\mathbb{H}$. The expansion up to second order in the coupling constant of this expression reveals that the ground state energy $\inf \sigma(\mathbb{H})$ can be divided into two terms: a `coherent term', given by $\inf_{(u,f)}\mathcal{E}(u,f)$, and a second term due to the contribution from the excited states of the electronic Hamiltonian. 

Our argument to prove the existence of a ground state relies on usual strategies from the calculus of variations. The main novelty comes from the possible absence of infrared and ultraviolet cutoffs in the interaction. This produces singular terms with a critical behavior  in the energy functional that we handle using, in particular, suitable estimates in Lorentz spaces. The use of weak versions of H\"older and Young's inequalities in Lorentz spaces seems to be new in the present context. It constitutes one of the main technical tools in our argument.

The ultraviolet convergence of the ground state and the asymptotic expansion of \eqref{eq:diff_en}  at small coupling also seem to be new. In order to establish them (as well as the uniqueness of the ground state), we project a non-linear eigenvalue equation associated to the minimization problem onto the vector space associated to the ground state of the electronic Hamiltonian and its orthogonal complement.

In a companion paper \cite{BreteauxFaupinPayet22}, we will study the Pauli-Fierz Hamiltonian of a non-relativistic particle with spin $\frac12$ coupled to the quantized radiation field in the standard model of non-relativistic QED. In this case, the quasi-classical energy coincides with the energy of a coupled Maxwell--Schr\"odinger system.

In the remainder of this section, we begin by introducing in Section \ref{subsec:elec} the abstract class of Hamiltonians we consider and our main hypotheses. Next, in Section \ref{subsec:linear}, we state our main results.

\medskip

\noindent \emph{Notations.}
We recall that for $1\le p<\infty$, the Lorentz spaces (or weak $L^p$ spaces) $L^{p,\infty}(\mathbb{R}^3)$ are defined as the set of (equivalence classes of) measurable functions $f:\mathbb{R}^3\to\mathbb{C}$ such that
\begin{equation}\label{eq:norm_weakLp}
\|f\|_{L^{p,\infty}}:=\sup_{t>0} \lambda \big( \{|f|>t\}\big) ^{\frac1p}\,t 
\end{equation}
is finite, where $\lambda$ denotes Lebesgue's measure.

The usual Fourier transform acting on tempered distribution is denoted by $\mathcal{F}$ with inverse~$(2\pi)^{-3}\bar{\mathcal F}$. (We use the normalization $\mathcal F(f)(x)=\int_{\mathbb{R}^3} e^{-ix\cdot\xi} f(\xi) \diff \xi$ for~$f$ in~$L^1(\mathbb{R}^3)$, and hence $\bar{\mathcal F}(f)(x)=\int_{\mathbb{R}^3} e^{ix\cdot\xi} f(\xi) \diff \xi$. This normalization is not the standard one but it is convenient in our context.) Throughout the paper, we use the following convention. Let $f,g$ be functions associated to tempered distributions. Assume that $\FF(g)$ identifies with a function such that~$f\FF(g)$ can be associated to a tempered distribution. We write
\begin{equation}\label{eq:convfourier}
\FF(f)*g := (2\pi)^{-3} \FF( f \bar{\FF}(g) ).
\end{equation}
This convention extends the well-known equality which holds e.g.~if $f$ and~$g$ are in~$L^1$ or $f$ is in~$L^2$ and $g$ in~$L^1$.

In several places, we will use localization functions in ${C}_0^\infty(\mathbb{R}^3)$ denoted by $\eta$ and such that~$0\le\eta\le1$, $\eta(x)=1$ if $|x|\le1$ and $\eta(x)=0$ if $|x|\ge2$. We define the non-negative function~$\tilde{\eta}$ by 
\begin{equation*}
\eta^2+\tilde{\eta}^2=1,
\end{equation*}
and for all $R>0$, we set 
\begin{equation}\label{eq:defetaR}
\eta_R(x):=\eta(x/R) \quad\text{and} \quad\tilde{\eta}_R(x):=\tilde{\eta}(x/R).
\end{equation}

If $\mathcal{H}_1$, $\mathcal{H}_2$ are two Hilbert spaces, $\mathcal{L}(\mathcal{H}_1,\mathcal{H}_2)$ stands for the set of bounded linear operators from $\mathcal{H}_1$ to $\mathcal{H}_2$. Given a linear operator $A$ on a Hilbert space $\mathcal{H}$, we denote by $\mathcal{D}(A)$ its domain and $\mathcal{Q}(A)$ its form domain.

\subsection{Model and assumptions}\label{subsec:elec}

Before defining the abstract class of linearly coupled Pauli-Fierz Hamiltonians we consider, we introduce our conditions on the electronic Hamiltonian~$H_{V}$.

\subsubsection{The electronic Hamiltonian}

We suppose that the non-relativistic particle is spinless and bound by an external potential. The Hilbert space and Hamiltonian for the particle are given by
\begin{equation}\label{eq:HV}
\mathcal{H}_{\mathrm{el}}:=L^2(\mathbb{R}^3), \quad H_V=-\Delta+V(x),
\end{equation}
where $V:\mathbb{R}^3\to\mathbb{R}$ is a real potential. We display the dependence on $V$ since one of our main hypotheses (see Hypothesis \ref{condV}) assumes the existence of a decomposition $V=V_1+V_2$ such that $V_1\ge0$, $V_2$ vanishes at $\infty$ and there is a gap between the ground state energies of $H_V$ and $H_{V_1}$.

The main examples we have in mind are confining potentials, $V(x)\to\infty$ as $|x|\to\infty$, and Coulomb-type potentials, $V(x)=-c|x|^{-1}$ with $c>0$. We introduce general hypotheses on~$V$ that are fulfilled by a large class of potentials, including the two preceding examples. As we will see below, some of our main results have interesting consequences in special cases, especially when $V$ is confining.

We set
\begin{equation*}
\mu_V:=\inf\sigma(H_{V}), 
\end{equation*}
and likewise if $V$ is replaced by another potential. For $U:\mathbb{R}^3\to\mathbb{R}$, we denote by
\begin{equation*}
U_+:=\max(U,0),\quad U_-:=\max(-U,0),
\end{equation*}
the positive and negative parts of $U$, respectively, so that $U=U_+-U_-$. %Recall that $\mu_{V_+}=\inf\sigma(H_{V_+})$.

We make the following hypothesis.
\begin{assumption}[Conditions on $V$]\label{condV}
There exist $0\le a<1$ and $b$ in~$\mathbb{R}$ such that the negative part of $V$ satisfies
\begin{equation*}
V_{-}\leq -a\Delta+b\,,
\end{equation*}
in the sense of quadratic forms on $H^1(\mathbb{R}^3)$. Moreover, $V$ decomposes as $V=V_{1}+V_{2}$ with
\begin{enumerate}[label=(\roman*)]
\item $V_{1}\in L_{\mathrm{loc}}^{1}(\mathbb{R}^{3};\mathbb{R}^{+})$, \vspace{0,2cm}% and $V_1-V_+ \in L^\infty + L^{3/2}$,
\item $V_{2}\in L_{\mathrm{loc}}^{3/2}(\mathbb{R}^{3};\mathbb{R})$ and ${\displaystyle \lim_{|x|\to\infty}V_{2}(x)=0}$,
\item $\mu_V<\mu_{V_1}$. 
\end{enumerate}
\end{assumption}
We have the following accompanying remarks (we refer to Section~\ref{subsec:Estimate-elec} for justifications). Since $V_+\ge0$, $H_{V_+}=-\Delta+V_+$ identifies with a non-negative self-adjoint operator on~$L^2(\mathbb{R}^3)$ with form domain
\begin{equation*}
\mathcal{Q}(H_{V_+}) = \mathcal{Q}(-\Delta)\cap \mathcal{Q}(V_+)=\Big \{ u \in H^1(\mathbb{R}^3) \mid \int_{\mathbb{R}^3} V_+(x) |u(x) |^2 \diff x < + \infty \Big \}.
\end{equation*} 
Moreover, it follows from our hypotheses that $H_V$ identifies with a semi-bounded self-adjoint operator with form domain 
\begin{equation*}
\mathcal{Q}_V:=\mathcal{Q}(H_V) = \mathcal{Q}(H_{V_+}) = \mathcal{Q}(H_{V_1}) \, ,
\end{equation*}
and that $\mathcal{Q}_V$ is a Hilbert space for the norm
\begin{equation}\label{eq:normQ}
\|u\|^2_{\mathcal{Q}_V}:=\|u\|^2_{H^1}+\big\|(V_+)^{\frac12}u\big\|^2_{L^2} \, .
\end{equation}
In particular, $\mu_V$ and $\mu_{V_1}$ are well-defined. 
We will most of the time consider a state $u$ in
\begin{equation}\label{eq:defU_intro}
\mathcal{U} := \{u\in\mathcal{Q}_V\mid \|u\|_{L^2}=1\} \,.
\end{equation}

In order to obtain uniqueness of minimizers, we require that the Schr\"odinger Hamiltonian~$H_V$ has a unique ground state.  By Perron-Froebenius arguments, it is well-known that, under suitable conditions on $V$, if $\mu_V$ is an eigenvalue of $H_V$ then it is simple and there exists a corresponding strictly positive eigenstate (see e.g.~\cite[Theorems XIII.46 and XIII.48]{ReedSimonII}). We make the following related hypothesis.
\begin{assumption}[Ground state of $H_V$]\label{condGS} The ground state energy $\mu_V$ of the particle Hamiltonian $H_V$ is a simple isolated eigenvalue associated to a unique positive ground state $u_V$ in~$L^2(\mathbb{R}^3;\mathbb{R}_+)$, such that $\|u_V\|_{L^2}=1$.
\end{assumption}

The orthogonal projection onto the vector space spanned by $u_V$ is denoted by~$\Pi_V$. We also set $\Pi_V^\perp:=\mathbf{I}-\Pi_V$.

\subsubsection{Linearly coupled Pauli-Fierz Hamiltonians}
We suppose that the radiation field is a scalar, bosonic field with Hilbert space given by the symmetric Fock space
\begin{equation}\label{eq:FockSpace}
\mathcal{H}_{\mathrm{f}}:=\F(L^2(\R^3))=\bigoplus_{n=0}^\infty\bigvee^nL^2(\R^3).
\end{equation}
In the momentum representation, the free field Hamiltonian is the second quantization of the multiplication operator by $\omega(k)$,
\begin{equation}
\mathbb{H}_{\mathrm{f}}:=\mathrm{d}\Gamma(\omega(k)),
\end{equation}
where $\omega:\mathbb{R}^3\to\mathbb{R}_{+}$ is a non-negative measurable function. See Appendix \ref{app:Fock} for the precise definition of second quantized operators. The coupling between the electron and the field is linear in the creation and annihilation operators, given by
\begin{equation}
\mathbb{H}_{\mathrm{int}}:=g\sqrt{2} \Phi(h_x),
\end{equation}
where $g$ in~$\mathbb{R}$ is a coupling constant, $\Phi(h)$, for $h$ in~$L^2(\mathbb{R}^3)$, denotes the field operator 
%defined by
%\begin{equation*}
%\Phi(h):=(a^*(h)+a(h))/\sqrt{2},
%\end{equation*}
(see Appendix \ref{app:Fock} for the definitions of the field, creation and annihilation operators% $a^\#(h)$
), and
\begin{equation}
 h_x(k):=v(k)e^{-ikx},
 \end{equation}
for all $x$ in~$\mathbb{R}^3$, where $v:\mathbb{R}^3\to\mathbb{R}$ is a coupling function. 

This framework covers several models of interest:
\begin{itemize}
\item The Nelson model \cite{Nelson64}, with the relativistic dispersion relation $\omega(k):=\sqrt{k^2+m^2}$ corresponding to a field of mass $m\ge0$, and the coupling function $v(k)=\omega(k)^{-\frac12}\chi(k)$ with $v$ in~$L^2(\mathbb{R}^3)$. Here, in particular,~$\chi$ incorporates an ultraviolet cutoff. Moreover, without infrared regularization, $\chi=1$ near $k=0$, while if an infrared regularization is imposed, one assumes that $\chi(0)=0$.
\item The Fr\"ohlich polaron model \cite{Frohlich39}, with $\omega(k)=1$ and $
v(k)=|k|^{-1}$.
\item The phonon Hamiltonian of solid state physic (see e.g.~\cite{Kittel63}), with a bounded dispersion relation  $\omega:\mathbb{R}^3\to\mathbb{R}$ such that $\omega(k)\sim c|k|$ near $k=0$ in the case of acoustic phonons, or~$0<c_1\le\omega(k)\le c_2$ in the case of optical phonons. Here $c,c_1,c_2$ are positive constants. Moreover, $v(k)=|k|^{\frac12}\chi(k)$, with $v$ in~$L^2(\mathbb{R}^3)$ and $\chi=1$ near $k=0$.
\end{itemize}

Assuming that $V$ satisfies Hypothesis \ref{condV} and that
\begin{equation}\label{eq:condv_sa}
\int_{\R^3}\frac{v^2}{\omega}%\omega(k)^{-1}|v(k)|^2\diff k
<\infty,
\end{equation}
it is not difficult to verify that, for all values of the coupling constant $g$, the total Hamiltonian
\begin{equation}\label{eq:general_PF}
\mathbb{H}=H_V\otimes\mathbf{I}_{\mathrm{f}}+\mathbf{I}_{\mathrm{el}}\otimes \mathbb{H}_{\mathrm{f}}+g\sqrt{2}\Phi(h_x),
\end{equation}
is a semi-bounded self-adjoint operator with form domain
\begin{equation}\label{eq:Hfree}
\mathcal{Q}(\mathbb{H})=\mathcal{Q}(\mathbb{H}_{\mathrm{free}}), \quad \mathbb{H}_{\mathrm{free}}:=H_V\otimes\mathbf{I}_{\mathrm{f}}+\mathbf{I}_{\mathrm{el}}\otimes \mathbb{H}_{\mathrm{f}}.
\end{equation}
See Appendix \ref{app:Fock} for details. The domains of $\mathbb{H}$ and $\mathbb{H}_\mathrm{free}$ in fact also coincide in this case. Note that the condition \eqref{eq:condv_sa} is satisfied in the case of the Nelson model and the phonon model, but not for the polaron model. In the latter case, one can still prove that $\mathbb{H}$ identifies with a self-adjoint operator with form domain $\mathcal{Q}(\mathbb{H})=\mathcal{Q}(\mathbb{H}_{\mathrm{free}})$, see \cite{GriesemerWunsch16}.

\subsubsection{Klein--Gordon--Schr\"odinger energy}

For $f$ in $L^2(\mathbb{R}^3)$, the coherent state of parameter $f$ is denoted by
\begin{equation*}
\Psi_f%:= e^{\mathrm{i}\Phi(f)}\Omega
:=e^{i\Phi(\frac{\sqrt{2}}{i} f )}\Omega \in\mathcal{H}_{\text{f}} \,,
%=e^{-\frac{\|f\|^{2}}{2}}\sum_{n=0}^{\infty}\frac{f^{\otimes n}}{\sqrt{n!}} \,,
\end{equation*}
where $\Omega$ stands for the Fock vacuum. Let $u$ in~$\mathcal{U}$ and let $f$ in~$L^2(\R^3)$ be such that
$\omega^{1/2}f$ belongs to~$L^2(\R^3).$ A simple computation shows that the energy of the product state $u\otimes \Psi_f$  is  given by
\begin{align}
 \big \langle (u\otimes \Psi_f),\mathbb{H}(u\otimes \Psi_f) \big \rangle_{\mathcal{H}} %\notag
=   \mathcal{E}(u,f) \label{eq:cE}
\end{align}
(see Appendix \ref{app:Fock}) where
\begin{align}
 \mathcal{E}(u,f)&:= \int_{\mathbb{R}^3}|\vec{\nabla} u(x)|^2\diff x+\int_{\mathbb{R}^3}V(x)|u(x)|^2\diff x+\int_{\mathbb{R}^3}\omega(k)|f(k)|^2\diff k\notag\\
&\quad+2g\,\Real\int_{\R^6}  e^{ikx}v(k)f(k) |u(x)|^2 \diff x\diff k \,. \label{eq:cE2}
\end{align}
Hence we obtain the energy of a coupled Klein--Gordon--Schr\"odinger system, the coupling being given by the last term in~\eqref{eq:cE2}.

We aim at proving the existence and uniqueness of a minimizer for the energy functional~$\mathcal{E}$, under suitable assumptions on $\omega$ and $v$.  
%Let us mention that a more general, natural question would be to minimize the energy of product states of the form $u\otimes\Psi$ with $\Psi$ a \emph{quasifree} state. Since the interaction Hamiltonian $\mathbb{H}_{\mathrm{int}}$ is linear in the creation and annihilation operators, however, it is not difficult to verify that the infimum of the energy over $u\otimes\Psi$ with $\Psi$ a quasifree state actually coincides with the infimum of the energy over $u\otimes\Psi$ with $\Psi$ a coherent state. Indeed, minimizing over quasifree state is the same as minimizing over pure quasifree states \cite{BachBreteauxKnoerrMenge14} and using properties of pure quasifree states~\cite[Prop.~4.5]{BachBreteauxKnoerrMenge14}, the result follows.
% it is easy to see that when minimizing over pure quasifree state the minimizer satisfies $\gamma=0$, and for pure bosonic quasifree states, $\gamma+\gamma^2=\alpha\alpha^*$, so that $\alpha=0$.
%\footnote{Ref: Generalized One-Particle Density Matrices and Quasifree States, Bach, Breteaux, Kn\"orr, Menge:}
%
The natural energy space for $\mathcal{E}(u,f)$ is $\mathcal{U}\times\mathcal{Z}_\omega$
where
\begin{equation*}
\ZZ_{\omega}:=\left\{f:\R^3\rightarrow\C\;\text{measurable} \mid \omega^{1/2}f\in L^2(\R^3, \diff k)\right\}.
\end{equation*}  
We make the following hypothesis on $v$ which, combined with Hypothesis \ref{condV}, ensures that~$\mathcal{E}$ is well-defined  on $\mathcal{U}\times\mathcal{Z}_\omega$ (see Proposition \ref{prop:KGS-Hartree} and Lemma \ref{lm:JV} below). Recall also that $v$ is real-valued and $\omega$ is supposed to be a non-negative measurable function. 
\begin{assumption}[Condition on $v$]\label{condv0:2a}
The map $W:=g^2\omega^{-1}v^2$ decomposes as $W=W_1+W_2$ with
\begin{enumerate}[label=(\roman*)]
\item $W_1\in L^1(\mathbb{R}^3)$,\vspace{0,1cm}
\item $W_2\in L^{3,\infty}(\mathbb{R}^3)$.
\end{enumerate}
\end{assumption}

It should be noted that this hypothesis covers all the examples previously mentioned (the Nelson, polaron and phonons Hamiltonians). Indeed, Hypothesis \ref{condv0:2a} is satisfied if one assumes that $\omega^{-1}v^2$ belongs to $L^1_{\mathrm{loc}}(\mathbb{R}^3)$ and that
 %, is of order $\mathcal{O}(|k|^{-3+\varepsilon})$ near $0$, and $\mathcal{O}(|k|^{-1})$ near $\infty$:
\begin{equation*}
|\omega^{-1}(k)v^2(k)|\le {C}_1|k|^{-3+\varepsilon}\mathds{1}_{|k|\le1}+{C}_2|k|^{-1}\mathds{1}_{|k|\ge1}, \quad \varepsilon>0,
\end{equation*}
for some positive constants ${C}_1$, ${C}_2$. This follows from the facts that $|k|^{-3+\varepsilon}\mathds{1}_{|k|\le1}$ is in~$L^1(\mathbb{R}^3)$ while $|k|^{-1}\mathds{1}_{|k|\ge1}$ is in~$L^{3,\infty}(\mathbb{R}^3)$. We emphasize in particular that, for the Nelson model, no infrared regularization is required and the ultraviolet cutoff can be removed, taking~$v(k)=|k|^{-1/2}$ and $\omega(k)=\sqrt{k^2+m^2}$, $m\ge0$.

\subsection{Main results}\label{subsec:linear} We are  now ready to state our main results. For clarity we decompose the presentation into a few subsections.

\subsubsection{Infrared problem for Klein--Gordon--Schr\"odinger}

We begin with a relatively simple property, which we refer to as the `infrared problem', keeping the usual terminology from QED. Since, in general, $\mathcal{Z}_\omega$ is not contained in $L^2(\mathbb{R}^3)$, the field component $f_{\mathrm{gs}}$ of a minimizer~$(u_{\mathrm{gs}}, f_{\mathrm{gs}})$ of the Klein--Gordon--Schr\"odinger energy functional over~$\mathcal{U}\times\mathcal{Z}_\omega$ may not belong to the original one-particle Hilbert space $\mathfrak{h}=L^2(\mathbb{R}^3)$. This formally corresponds to the fact that the coherent state $\Psi_{f_{\mathrm{gs}}}$ does not belong to Fock space. On the other hand, if $f_{\mathrm{gs}}$ belongs to~$L^2(\mathbb{R}^3)$, then, using in addition that $u_{\mathrm{gs}}$ belongs to~$\mathcal{Q}_V$ and $f_{\mathrm{gs}}$ to~$\mathcal{Z}_\omega$, one easily verifies that~$u_{\mathrm{gs}}\otimes\Psi_{f_\mathrm{gs}}$ belongs to~$\mathcal{Q}(\mathbb{H})$, so that the Pauli-Fierz energy \eqref{eq:cE} in the state $u_{\mathrm{gs}}\otimes\Psi_{f_\mathrm{gs}}$ is well-defined.

The next proposition provides both a necessary and a sufficient condition ensuring that~$f_{\mathrm{gs}}$ belongs to $L^2(\mathbb{R}^3)$. 
\begin{proposition}\label{prop:condfL2}
Suppose that $V$ satisfies Hypothesis~\ref{condV} and~$W$ satisfies Hypothesis~\ref{condv0:2a}. If~$(u_{\mathrm{gs}},f_{\mathrm{gs}})$ is a minimizer of the Klein--Gordon--Schr\"odinger energy functional over~$\mathcal{U}\times\mathcal{Z}_\omega$\,, then
\begin{equation*}
\frac{v}{\omega} \mathds{1}_{|k|\le1}+\frac{v}{|k|\omega} \mathds{1}_{|k|\ge1} \in L^2(\R^3) \, \Longrightarrow \, f_{\mathrm{gs}}\in L^2(\mathbb{R}^3) \, \Longrightarrow \,  \frac{v}{\omega} \mathds{1}_{|k|\le1} \in L^2(\R^3).
\end{equation*}
\end{proposition}
Considering the massless Nelson model where $\omega(k)=|k|$ and $v(k)=|k|^{-1/2}\chi(k)$, the previous conditions reduce to 
\begin{equation*}
k\mapsto |k|^{-5/2}\chi(k) \mathds{1}_{|k|\ge1} \in L^2(\R^3)\quad\text{and}\quad k\mapsto |k|^{-3/2}\chi(k) \mathds{1}_{|k|\le1} \in L^2(\R^3).
\end{equation*}
Thus, in order to have that $f_{\mathrm{gs}}$ belongs to~$L^2(\R^3)$, it is necessary to impose an infrared regularization, but no ultraviolet regularization is needed. Note that the presence of an infrared regularization is also necessary to have the existence of a ground state for the massless Nelson Hamiltonian \cite{LorincziMinlosSpohn02,DerezinskiGerard04,Gerard03}, while it is well-known that the Nelson Hamiltonian is renormalizable in the ultraviolet limit \cite{Nelson64}.

For the Fr\"ohlich polaron model, we have $\omega=1$, hence $\mathcal{Z}_\omega=L^2(\mathbb{R}^3)$ and the previous proposition is trivial.

\subsubsection{Existence of a ground state} One of our main results is the following theorem which provides the existence of a ground state for the Klein--Gordon--Schr\"odinger energy functional under our general assumptions on  $V$ and $W$.
\begin{theorem}[Existence of a ground state]\label{thm:mainlnc}
Suppose that $V$ satisfies Hypothesis \ref{condV} and $W$ satisfies Hypothesis \ref{condv0:2a}. 
There exists $C_V>0$ such that, if the decompositions $V=V_1+V_2$ and~$W=W_1+W_2$ as in Hypotheses \ref{condV} and \ref{condv0:2a}, respectively, can be chosen such that
\begin{equation}\label{eq:condv:0}
\|W_1\|_{L^1}+C_V\|W_2\|_{L^{3,\infty}}\le \delta(\mu_{V_1}-\mu_V)
\end{equation}
and
\begin{equation}\label{eq:smallness}
C\|W_2\|_{L^{3,\infty}}\le \frac12(1-a),
\end{equation}
for some universal constants $C,\delta>0$ and where $a$ is given by Hypothesis \ref{condV}, then the Klein--Gordon--Schr\"odinger energy functional \eqref{eq:cE2} has a minimizer over~$\mathcal{U}\times\mathcal{Z}_\omega$.
\end{theorem}
We have the following accompanying remarks concerning the smallness conditions \eqref{eq:condv:0} and \eqref{eq:smallness}.

\begin{remark}\label{rk:intro_Nelson2}
The smallness condition \eqref{eq:smallness} only concerns the term $W_2$ in~$L^{3,\infty}$ of $W$, not the term $W_1$ in~$L^1$. Moreover, in the case of a confining potential, $V(x)\to\infty$ as $|x|\to\infty$, the condition \eqref{eq:condv:0} is automatically satisfied provided one suitably chooses the potential $V_1$, see Lemma \ref{lm:confining_intro}. This implies that if $V$ is confining and $W_2=0$, then a minimizer exists for any~$g$ in~$\R$. 
In fact, for the special case of a  confining potential, one can prove the existence of a minimizer by simpler arguments than those we use in the proof of Theorem \ref{thm:mainlnc}, since in this case the relative compactness of minimizing sequences can easily be deduced from the confining assumption. 
%The advantage of the hypotheses we make in Theorem \ref{thm:mainlnc} is that they allow us to cover both settings of a binding, vanishing potential and of a confining potential.}
%\end{enumerate}
\end{remark}
\begin{remark}\label{rk:intro_Nelson1}%$ $
%\begin{enumerate}[label=(\roman*)]
%\item 
Our assumptions cover the critical case $\bar{\FF}(W)(x)=g^2|x|^{-2}$ of the Hartree equation \eqref{eq:Hartree_eq} (taking $W(k)=cg^2|k|^{-1}$ in~$L^{3,\infty}(\mathbb{R}^3)$), which has been studied e.g.~in \cite{GuoSeiringer14,DengLuShuai15}. In particular, with $\bar{\FF}(W)(x)=g^2|x|^{-2}$, it has been proven in \cite{GuoSeiringer14,DengLuShuai15} that the Hartree energy has no minimizer for $g$ larger than some critical value $g^*$. Hence the smallness condition \eqref{eq:smallness} in Theorem \ref{thm:mainlnc} cannot be removed.
%\item 
\end{remark}

The proof of Theorem \ref{thm:mainlnc} follows from observing that $(u,f)$ is a minimizer for the Klein--Gordon--Schr\"odinger energy functional \eqref{eq:cE2} if and only if it is of the form $(u,f_u)$ where the field parameter satisfies~$f_u=-g\omega^{-1}v\FF(|u|^2)$ and where $u$ minimizes the \emph{Hartree energy}
\begin{align}\label{eq:Hartree_eq} 
 J(u)= \langle u,H_Vu\rangle_{L^2_x}-\int_{\mathbb{R}^3} \big(\bar{\FF}(W)*|u|^2\big)(x) |u(x)|^2 \diff x, \quad u \in \mathcal{U}.
\end{align}
Our strategy then rests on usual arguments from the calculus of variations \cite{LionsI84,LionsII84}. Existence of minimizers for the Hartree energy has been studied by many authors in different contexts (see e.g.~\cite{AmmariFalconiOlivieri21,AFGST02,GuoSeiringer14,DengLuShuai15,FrohlichLenzmann04,LewinNamRougerie14,Lieb77, LionsI84,LionsII84,GriesemerHantschWellig12} and references therein). We are not aware, however, of a result giving the existence of a minimizer under our general conditions on $V$ and $W$. The main difficulties come from the fact that we consider external potentials with possibly both a confining and a negative part, the latter vanishing at infinity, and, more importantly, that our assumptions on the convolution term in the Hartree energy \eqref{eq:Hartree_eq} concerns the Fourier transform of the usual pair potential, with possibly a critical behavior corresponding to the term $W_2$ in~$L^{3,\infty}$. Such critical terms are due to the fact that we do not impose an ultraviolet cutoff into the interaction. To handle them, we have to rely on suitable estimates in Lorentz spaces whose use, to our knowledge, seems to be new in the context of minimizing the Hartree energy functional.  For completeness, we provide a complete proof of the existence of a minimizer for \eqref{eq:Hartree_eq} under our conditions in Appendix~\ref{app:Hartree}.

We mention that the minimization problem for $\mathcal{E}(u,f)$ has been studied in the recent paper \cite{AmmariFalconiOlivieri21}, in the particular case of the massive Nelson model with $V$ confining, the dispersion relation $\omega(k)=\sqrt{k^2+m^2}$ and $v(k)=\omega(k)^{-1/2}\chi(k)$ with $\chi$ a smooth compactly supported function. Our results cover this particular case.

\subsubsection{Uniqueness of the ground state and expansion of the ground state energy at small coupling} Our next concern is the question of the uniqueness of the ground state for the Klein--Gordon--Schr\"odinger energy functional. To establish it, we need to strengthen our assumptions, assuming that the electronic Hamiltonian $H_V$ has a unique ground state as stated in Hypothesis \ref{condGS} and that the coupling is sufficiently small. Of course, uniqueness of a minimizer for $\mathcal{E}(u,f)$ only holds up to a phase, since $\mathcal{E}(u,f)=\mathcal{E}(e^{i\theta}u,f)$ for any $\theta$ in $\mathbb{R}$.
\begin{theorem}[Uniqueness of the ground state]\label{thm:mainlnc2}
Suppose that $V$ satisfies Hypotheses~\ref{condV} and \ref{condGS} and that $W$ satisfies Hypothesis \ref{condv0:2a}. There exists $\varepsilon_V>0$ such that, if
\begin{equation*}
\|W\|_{L^1+L^{3,\infty}}\le\varepsilon_V,
\end{equation*}
then the Klein--Gordon--Schr\"odinger energy functional \eqref{eq:cE2} has a unique minimizer $(u_{\mathrm{gs}},f_{\mathrm{gs}})$ in~$\mathcal{U}\times\ZZ_{\omega}$ such that $\langle u_{\mathrm{gs}},u_V\rangle_{L^2}>0$. 
\end{theorem}
Under the conditions of the previous theorem, recalling that $W=g^2\omega^{-1}v^2$, we can now compute the asymptotic expansion of the ground state energy as the coupling constant $g$ goes to $0$.
\begin{proposition}[Expansion of the ground state energy at small coupling]\label{thm:mainlnc3}
Under the conditions of Theorem \ref{thm:mainlnc2}, we have
\begin{align}
\underset{(u,f)\in\, \mathcal{U}\times\ZZ_{\omega}}{\min}\E(u,f)%&=\E(u_{\mathrm{gs}},f_{\mathrm{gs}}) \notag\\
&=\mu_V-g^2\int_{\R^3} \big(\bar{\FF}(\omega^{-1}v^2)*\abs{u_V}^2\big)(x)\abs{u_V(x)}^2\diff x+\mathcal{O}(g^4),\label{asympt_GS}
\end{align}
as $g\to0$.
\end{proposition}
To obtain uniqueness of the minimizer, as well as the expansion \eqref{asympt_GS}, we use that any minimizer of \eqref{eq:Hartree_eq} is a non-linear Hartree eigenstate and project the non-linear eigenvalue equation to the vector space spanned by the electronic ground state $u_V$ and its orthogonal complement.

%

%\subsubsection{Comparison between the ground state energies of the Klein--Gordon--Schr\"odinger energy functional $\mathcal{E}$ and of the Pauli-Fierz Hamiltonian $\mathbb{H}$}

As mentioned above, in the case where $W_2=0$, i.e.~$W=g^2\omega^{-1}v^2$ is in~$L^1(\mathbb{R}^3)$, the Hamiltonian $\mathbb{H}$ in \eqref{eq:general_PF} identifies with a semi-bounded self-adjoint operator. Hence we can compare the ground state energy of $\mathbb{H}$ with its quasi-classical counterpart:
\begin{proposition}[Comparison with the ground state energy of $\mathbb{H}$]\label{prop:diff_GSE}
Under the conditions of Theorem \ref{thm:mainlnc2}, with $W$ in~$L^1(\mathbb{R}^3)$, we have
\begin{multline*}
\inf \sigma(\mathbb{H})-\mathcal{E}(u_{\mathrm{gs}},f_{\mathrm{gs}})\\
=-g^2 \int_{\mathbb{R}^3} v(k)^2 \big\langle u_V , e^{ikx} \Pi_V^\perp\big({H}_V-\mu_V+|k|\big)^{-1} \Pi_V^\perp e^{-ikx} u_V\big\rangle_{L^2_x}\mathrm{d}k +o(g^2) \,,
\end{multline*}
as $g\to 0$.
\end{proposition}
The term of order $g^2$ of the asymptotic expansion given by Proposition \ref{prop:diff_GSE} can be rewritten as
\begin{multline*}
\inf \sigma(\mathbb{H})-\mathcal{E}(u_{\mathrm{gs}},f_{\mathrm{gs}})\\
=-g^2 \ps{u_V\otimes\Omega}{a(h_x)(\Pi_V^\perp\otimes\mathbf{I}_{\mathrm{f}})\big(\mathbb{H}_{\mathrm{free}}-\mu_V\big)^{-1}(\Pi_V^\perp\otimes\mathbf{I}_{\mathrm{f}})a^*(h_x)u_V\otimes\Omega} +o(g^2) \,.
\end{multline*}
It should be compared with the term of order $g^2$ in the asymptotic expansion \eqref{asympt_GS} of the quasi-classical ground state energy $\E(u_{\mathrm{gs}},f_{\mathrm{gs}})$, which is given by
\begin{align}
&-g^2 \int_{\R^3} \big(\bar{\FF}(\omega^{-1}v^2)*\abs{u_V}^2\big)(x)\abs{u_V(x)}^2\diff x \notag\\
&=-g^2\ps{u_V\otimes\Omega}{a(h_x)(\Pi_V\otimes\Pi_\Omega^\perp)\big(\mathbb{H}_{\mathrm{free}}-\mu_V\big)^{-1}(\Pi_V\otimes\Pi_\Omega^\perp)a^*(h_x)u_V\otimes\Omega}, \label{eq:comput-2ndorder}
\end{align}
where $\Pi_\Omega$ is the projection onto the Fock vacuum and $\Pi_\Omega^\perp:=\mathbf{I}-\Pi_\Omega$. Hence we see that, at second order in the coupling constant, the ground state energy of $\mathbb{H}$ can be divided into two terms: a `coherent' term which is independent of the excited electronic eigenstates, and a `non-coherent' term which sums the contributions from these excited states. In particular, defining $\delta_V:=\mathrm{dist}(\mu_V,\sigma(H_V)\setminus\{\mu_V\})$ the distance between $\mu_V$ and the rest of the spectrum of $H_V$, we deduce from the previous expressions that if $\delta_V$ is large, then the non-coherent term is small and hence the coherent term becomes a good approximation to the ground state energy of $\mathbb{H}$. 

\subsubsection{Ultraviolet limit}

We suppose here that the coupling function is cut-off in the ultraviolet, i.e.~that it is of the form $v_\Lambda=v\mathds{1}_{|k|\le\Lambda}$ for some ultraviolet parameter $0<\Lambda<\infty$. We are interested in the ultraviolet limit $\Lambda\to\infty$. We write
\begin{equation*}
W_\Lambda:=g^2\omega^{-1}v_\Lambda=W\mathds{1}_{|k|\le\Lambda},
\end{equation*}
and note that if $W$ satisfies Hypothesis \ref{condv0:2a}, then for all $\Lambda>0$, $W_\Lambda$ is in~$L^1$ (this follows from the weak H\"older inequality, see \eqref{eq:weak_Holder} %\eqref{eq:weak_Holder2}
below). The fact that $W_\Lambda$ belongs to~$L^1$  in turn ensures that the Pauli-Fierz Hamiltonian 
\begin{equation*}
\mathbb{H}_\Lambda:=H_V\otimes\mathbf{I}_{\mathrm{f}}+\mathbf{I}_{\mathrm{el}}\otimes \mathbb{H}_{\mathrm{f}}+g\Phi(h_{\Lambda,x}),\quad h_{\Lambda,x}(k):=v_\Lambda(k)e^{-ikx},
\end{equation*}
identifies to a self-adjoint operator (see Appendix \ref{app:Fock}). 

%We suppose that
%\begin{equation*}
%W_\infty := g^2\omega^{-1}v^2 \in L^1(\mathbb{R}^3)+L^{3,\infty}(\mathbb{R}^3).
%\end{equation*}

Let $\mathcal{E}_\Lambda$ be the Klein--Gordon--Schr\"odinger energy functional with an ultraviolet cutoff, i.e.~$\mathcal{E}_\Lambda$ is given by \eqref{eq:cE2} with $v_\Lambda$ instead of $v$. The next proposition establishes the convergence of the ground state energies in the ultraviolet limit. Note that the assumptions are rather weak. In particular they do not necessarily imply the existence of a ground state for~$\mathcal{E}$ and $\mathcal{E}_\Lambda$.
\begin{proposition}[Ultraviolet limit of the ground state energies]\label{prop:conv-GS-en-intro}
Suppose that $V$ satisfies Hypothesis \ref{condV} and that $W$ satisfies Hypothesis \ref{condv0:2a}. Then
\begin{equation*}
\underset{(u,f)\in\, \mathcal{U}\times\ZZ_{\omega}}{\inf}\E_\Lambda(u,f) \underset{\Lambda\to\infty}{\longrightarrow} \underset{(u,f)\in\, \mathcal{U}\times\ZZ_{\omega}}{\inf}\E(u,f).
\end{equation*}
\end{proposition}

Under conditions ensuring that $\E_\Lambda$ and $\E$ have unique minimizers, we can also establish the convergence of the ground states of $\E_\Lambda$ to the ground state of $\E$, as $\Lambda\to\infty$. 
\begin{proposition}[Ultraviolet limit of the ground states]\label{prop:conv-GS2-en-intro}
Suppose that $V$ satisfies Hypotheses~\ref{condV} and \ref{condGS} and that $W$ satisfies Hypothesis \ref{condv0:2a}. There exists $\varepsilon_V>0$ such that, if
\begin{equation*}
\|W\|_{L^1+L^{3,\infty}}\le\varepsilon_V,
\end{equation*}
then for all $\Lambda>0$, $\E_\Lambda$ and $\E$ have unique minimizers $(u_{\Lambda,\mathrm{gs}},f_{\Lambda,\mathrm{gs}})$ and $(u_{\mathrm{gs}},f_{\mathrm{gs}})$ in $\mathcal{U}\times\mathcal{Z}_\omega$, respectively, such that $\langle u_{\Lambda,\mathrm{gs}},u_V\rangle_{L^2}>0$ and $\langle u_{\mathrm{gs}},u_V\rangle_{L^2}>0$. They satisfy
\begin{equation*}
\big\|(u_{\Lambda,\mathrm{gs}},f_{\Lambda,\mathrm{gs}}) - (u_{\mathrm{gs}},f_{\mathrm{gs}}) \big\|_{\mathcal{Q}_V\times\mathcal{Z}_\omega} \underset{\Lambda\to\infty}{\longrightarrow} 0.
\end{equation*}
\end{proposition}
The proofs of Propositions \ref{prop:conv-GS-en-intro} and \ref{prop:conv-GS2-en-intro} are not straightforward. The main difficulty comes from the fact that, in general, $W_\Lambda$ does \emph{not} converge to $W$ in $L^1+L^{3,\infty}$. To circumvent this difficulty, we use a convergence property in a weaker sense, based on a suitable application of Lebesgue's dominated convergence theorem.

\subsection{Organisation of the paper}
Our paper is essentially self-contained. It is organized as follows. Section \ref{sec:prelim} is a preliminary section containing several technical estimates that we subsequently use in Section~\ref{sec:linear} to establish our main results. In Appendix \ref{app:Fock}, we recall the definitions of standard objects related to second quantization as well as the self-adjointness of the Pauli-Fierz Hamiltonian~$\mathbb{H}$. Appendix \ref{app:Hartree} contains a proof of the existence of a minimizer for the Hartree energy functional under our conditions.

\medskip

\noindent \textbf{Acknowledgements.} The research of S.B.~was partly done during a CNRS sabbatical semester.

\section{Preliminaries}\label{sec:prelim}

In this preliminary section, we gather several technical estimates that are useful for our concern. The first subsection mainly concerns the electronic Hamiltonian $H_V$. In a second subsection, we prove some functional estimates in Lorentz spaces that are used in Section \ref{sec:linear} in a crucial way to control the interactions terms of the Klein--Gordon--Schr\"odinger energy functional.

\subsection{Estimates on the electronic part}\label{subsec:Estimate-elec}

Recall that our assumptions on the external potential $V$ of the electronic Hamiltonian $H_V=-\Delta+V$ have been introduced in Section \ref{subsec:elec}. We begin with a few remarks showing that $H_V$ is well-defined and that its form domain satisfies $\mathcal{Q}(H_V)=\mathcal{Q}(H_{V_+}) = \mathcal{Q}(H_{V_1})$, with $V_1$ as in Hypothesis \ref{condV}.

First, $V_-$ is form bounded with respect to $-\Delta$ with a relative bound less than $1$, by Hypothesis \ref{condV}. Thus $V_-$ is also form bounded with respect to $H_{V_+}$ with a relative bound less than $1$, and hence the KLMN Theorem (see e.g.~\cite[Theorem X.17]{ReedSimonII}) implies that $H_V$ identifies with a semi-bounded self-adjoint operator with form domain $\mathcal{Q}(H_V) = \mathcal{Q}(H_{V_+})$.

Next, Hypothesis \ref{condV}(ii) implies that $V_2$ is relatively form bounded with respect to $-\Delta$ with relative bound $0$. Indeed, for $R$ large enough, we have $V_2\mathds{1}_{|x|\ge R} \in L^\infty(\mathbb{R}^3)$ since $V_2(x)\to0$ as $|x|\to\infty$, while $V_2\mathds{1}_{|x|\le R}\in L^{3/2}(B_R)$ with $B_R:=\{x\in\mathbb{R}^3\, | \, |x|\le R\}$, since $V_2\in L^{3/2}_{\mathrm{loc}}(\mathbb{R}^3)$. Therefore $V_2\in L^{3/2}(\mathbb{R}^3) + L^\infty(\mathbb{R}^3)$ and hence we can apply \cite[Theorem X.19]{ReedSimonII} to deduce that $V_2$ is infinitesimally form-bounded with respect to $-\Delta$.  In turn, since $V_+-V_1=V_2+V_-$ is form bounded with respect to $-\Delta$, it is not difficult to verify that $\mathcal{Q}(H_{V_+}) = \mathcal{Q}(H_{V_1})$.

Recall the notation $\mathcal{Q}_V=\mathcal{Q}(H_V)$. We begin with the following easy lemma.
\begin{lemma}\label{lm:relative-bound}
Suppose that $V$ satisfies Hypothesis \ref{condV}. Then, for all $u$ in~$\mathcal{Q}_V$,
\begin{equation}\label{eq:estimH1}
\|u\|_{\dot{H}^1}^2\le\frac{1}{1-a}\big(\langle u,H_Vu\rangle+b\|u\|_{L^2}^2\big).
\end{equation}
\end{lemma}
\begin{proof}
The positivity of $V_+$ and the bound on $V_-$ from Hypothesis \ref{condV} yield, for $u$ in~$\mathcal{Q}_V$,
\begin{multline*}
\langle u , H_V u \rangle \ge \|u\|^2_{\dot{H}^1}-\langle u,V_-u\rangle 
\ge \|u\|^2_{\dot{H}^1}-a\langle u,-\Delta u\rangle-b\|u\|_{L^2}^2 = (1-a)\|u\|^2_{\dot{H}^1}-b\|u\|_{L^2}^2,
\end{multline*}
which proves the result.
\end{proof}

The next lemma shows that, for confining potentials~$V$, the gap $\mu_{V_1}-\mu_V$ can be made as large as we want, provided that the potential $V_1$ is suitably chosen. 

\begin{lemma}\label{lm:confining_intro}
Suppose that $V=V_+-V_-$ is such that
\begin{enumerate}[label=(\roman*)]
\item $V_+\in L^1_{\mathrm{loc}}(\mathbb{R}^3)$, \vspace{0,1cm}
\item $V_-\in L^{3/2}_{\mathrm{loc}}(\mathbb{R}^3)$, \vspace{0,1cm}
\item $V(x)\to\infty$ as $|x|\to\infty$.
\end{enumerate}
Then, for all $C>0$, there exist a decomposition $V=V_{1,C}+V_{2,C}$ as in Hypothesis~\ref{condV} such that, moreover,
\begin{equation*}
\mu_{V_{1,C}}-\mu_V\ge C.
\end{equation*}
\end{lemma}

\begin{proof}%[Proof of Lemma \ref{lm:confining_intro}]
Recall that the localizations functions $\eta_R$, $\tilde{\eta}_R$ have been defined in \eqref{eq:defetaR}.
%We show that one can find a decomposition $V=V_{1,C}+V_{2,C}$ satisfying Hypothesis~\ref{condV} and $\mu_{V_{1}}-\mu_V>C$.  Given $\alpha>0$ and $R>0$, 
Let $C>0$. We set
\begin{equation*}
V_{1,C}=V_{+}+2C\eta_R^2, \quad V_{2,C}=-V_{-}-2C\eta_R^2.
\end{equation*}
Observe that $V_{1,C}+V_{2,C}=V_+-V_-=V$. Moreover, since $V(x)\to\infty$ as $|x|\to\infty$, we have that $V_-(x)=0$ for $|x|$ large enough. Hence, since in addition $\eta_R^2$ is smooth and compactly supported, one sees that the decomposition $V=V_{1,C}+V_{2,C}$ satisfies the conditions of Hypothesis~\ref{condV} for any~$R$.

Now we verify that $\mu_{V_{1,C}}-\mu_V>C$ for suitably chosen $R$. Using the IMS localization formula (see e.g.~\cite{CFKS87}), we write
\begin{align*}
\mu_{V_{1,C}} & =\inf_{u\in \mathcal{U}}\Big(\langle(\eta_{R}^{2}+\tilde{\eta}_{R}^{2})u,(-\Delta+V_+)u\rangle+2C\|\eta_{R}u\|_{L^2}^{2}\Big)\\
 & =\inf_{u\in \mathcal{U}}\Big(\langle\eta_{R}u,(-\Delta+V_{+})\eta_{R}u\rangle+\langle\tilde{\eta}_{R}u,(-\Delta+V_{+})\tilde{\eta}_{R}u\rangle+o(R^0)+2C\|\eta_{R}u\|_{L^2}^{2}\Big),
 \end{align*}
since $|\vec{\nabla}\eta_{R}|^{2}+|\vec{\nabla}\tilde{\eta}_{R}|^{2}=o(R^0)$, $R\to\infty$. Next, using that $-\Delta\ge0$ and that $\mathrm{supp}(\tilde{\eta}_R)\subset B(0,R)^c$, we estimate
 \begin{align*}
 \mu_{V_{1,C}} &  \geq\inf_{u\in \mathcal{U}}\Big(\big( \mu_{V_{+}} +2C\big )\|\eta_{R}u\|_{L^2}^{2}+\Big(\inf_{x\in B(0,R)^{c}}V_{+}(x) \Big)\|\tilde{\eta}_{R}u\|_{L^2}^{2}+o(R^0)\Big).
\end{align*}
Since $V_+(x)\to\infty$ as $|x|\to\infty$, there exists $R_{0}>0$ such that for $R\geq R_{0}$, 
\begin{equation*}
\inf_{x\in B(0,R)^{c}}V_{+}(x)\geq \mu_{V_{+}}+2C.
\end{equation*}
Therefore, for $R\geq R_{0}$, we obtain
\begin{align}\label{eq:estim1}
\mu_{V_{1,C}} & \geq\inf_{u\in\mathcal{U}}\Big(\big(\mu_{V_{+}}+2C\big)\big(\|\eta_{R}u\|_{L^2}^{2}+\|\tilde{\eta}_{R}u\|_{L^2}^{2}\big)+o(R^0)\Big) = \mu_{V_{+}}+2C+o(R^0).
\end{align}
On the other hand, since $V_-\ge0$, we have that
\begin{equation}\label{eq:estim2}
\mu_{V_{+}}\geq\inf_{u\in\mathcal{U}}\langle u,(-\Delta+V_{+}-V_{-})u\rangle=\mu_{V}.
\end{equation}
Combining \eqref{eq:estim1} and \eqref{eq:estim2} gives
\begin{equation*}
\mu_{V_{1,C}}\geq \mu_{V}+2C+o(R^0).
\end{equation*}
Fixing $R$ large enough, we deduce that $\mu_{V_{1,C}}-\mu_V>C$, which proves the lemma.
\end{proof}

To conclude this section, we give a lemma which is useful to prove the existence of minimizers for the energy functional studied in Section \ref{sec:linear}.

\begin{lemma}
\label{lem:LowerSemiContDeltaV}Suppose that $V$ satisfies Hypothesis~\ref{condV}. Let $(u_{j})_{j\in\mathbb{N}}$ be a bounded sequence in~$H^{1}(\mathbb{R}^{3})$ which  converges weakly to $u_{\infty}$ in $H^1(\mathbb{R}^3)$, and strongly in~$L^{2}(\mathbb{R}^{3})$. Then 
\begin{equation*}
\langle u_{\infty},(-\Delta+V)u_{\infty}\rangle\leq\liminf_{j\to\infty}\langle u_{j},(-\Delta+V)u_{j}\rangle\,.
\end{equation*}
\end{lemma}

\begin{proof}
We consider each term of
\begin{equation}
\langle u,-\Delta u\rangle+\langle u,V_{1}u\rangle+\langle u,V_{2}u\rangle\label{eq:d0a}
\end{equation}
separately.

The first one is handled using the lower semi-continuity of $\|\cdot\|_{L^2}$.
Indeed, as~$u_{j}\to u_{\infty}$ weakly in $H^{1}(\mathbb{R}^{3})$,
it follows that $\nabla u_{j}\to\nabla u_{\infty}$ weakly in $L^{2}(\mathbb{R}^{3})$
and hence 
\begin{equation}
\langle u_{\infty},-\Delta u_{\infty}\rangle=\|\nabla u_{\infty}\|_{L^2}^{2}\le\liminf_{j\to\infty}\|\nabla u_{j}\|_{L^2}^{2}=\liminf_{j\to\infty}\langle u_{j},-\Delta u_{j}\rangle.\label{eq:d1}
\end{equation}
For the second term of \eqref{eq:d0a}, we use Fatou's Lemma, which
gives, since $V_{1}\ge0$, 
\begin{equation}
\langle u_{\infty},V_{1}u_{\infty}\rangle\le\liminf_{j\to\infty}\langle u_{j},V_{1}u_{j}\rangle.\label{eq:d2}
\end{equation}
As for the third term in \eqref{eq:d0a}, we claim that 
\begin{equation}
\langle u_{\infty},V_{2}u_{\infty}\rangle=\lim_{j\to\infty}\langle u_{j},V_{2}u_{j}\rangle,\label{eq:d3}
\end{equation}
for some suitable subsequence that we keep denoting by $(u_{j})_{j\in\mathbb{N}}$.
Indeed, let $\varepsilon>0$. We have that $\|\mathds{1}_{|x|>R_{0}}V_{2}\|_{\infty}\le\varepsilon$
for $R_{0}$ large enough, since $V_{2}(x)\to0$ as $|x|\to\infty$.
Therefore, for all~$j$ in~$\mathbb{N}$, 
\begin{equation}
\langle u_{j},\mathds{1}_{|x|>R_{0}}V_{2}u_{j}\rangle\le\varepsilon,\quad\langle u_{\infty},\mathds{1}_{|x|>R_{0}}V_{2}u_{\infty}\rangle\le\varepsilon.\label{eq:d4}
\end{equation}
Next, we approximate $\mathds{1}_{|x|\le R_{0}}V_{2}$ by a more regular
function. More precisely, since $\mathds{1}_{|x|\le R_{0}}V_{2}$ lies in~$L^{3/2}(\mathbb{R}^{3})$,
one can find $V_{2,\varepsilon}$ in~${C}_{0}^{\infty}(\mathbb{R}^{3})$
such that 
\[
\big\|\mathds{1}_{|x|\le R_{0}}V_{2}-V_{2,\varepsilon}\big\|_{L^{3/2}}\le\varepsilon.
\]
H{\"o}lder's inequality together with Sobolev's embedding $H^{1}(\mathbb{R}^{3})\subset L^{6}(\mathbb{R}^{3})$
then yield 
\begin{align}
\big|\langle u_{j},\mathds{1}_{|x|\le R_{0}}V_{2}u_{j}\rangle-\langle u_{j},\mathds{1}_{|x|\le R_{0}}V_{2,\varepsilon}u_{j}\rangle\big|&\le\big\|\mathds{1}_{|x|\le R_{0}}V_{2}-V_{2,\varepsilon}\big\|_{L^{3/2}}\|u_{j}\|_{L^6}^{2}\notag\\
&\lesssim\varepsilon\|u_{j}\|_{H^{1}}^{2}\lesssim\varepsilon,\label{eq:d43}
\end{align}
since we assumed that $(u_{j})_{j\in\mathbb{N}}$ is bounded in $H^{1}(\mathbb{R}^{3})$.
Likewise, 
\begin{align}
\big|\langle u_{\infty},\mathds{1}_{|x|\le R_{0}}V_{2}u_{\infty}\rangle-\langle u_{\infty},\mathds{1}_{|x|\le R_{0}}V_{2,\varepsilon}u_{\infty}\rangle\big|\lesssim\varepsilon.\label{eq:d44}
\end{align}
Now, since $\mathds{1}_{|x|\le R_{0}}u_{j}\to\mathds{1}_{|x|\le R_{0}}u_{\infty}$
strongly in $L^{2}(\mathbb{R}^{3})$, and since $V_{2,\varepsilon}$
is bounded, we deduce that 
\begin{equation}
\langle u_{\infty},\mathds{1}_{|x|\le R_{0}}V_{2,\varepsilon}u_{\infty}\rangle=\lim_{j\to\infty}\langle u_{j},\mathds{1}_{|x|\le R_{0}}V_{2,\varepsilon}u_{j}\rangle.\label{eq:d5}
\end{equation}
Combining \eqref{eq:d4}, \eqref{eq:d43}, \eqref{eq:d44} and \eqref{eq:d5},
we obtain \eqref{eq:d3}.
\end{proof}

\subsection{Some functional inequalities in Lorentz spaces}

In the proof of our main results, we will use in a crucial way some functional inequalities in Lorentz spaces that we present in this section. For $1\le p<\infty$, the Lorentz spaces $L^{p,\infty}=L^{p,\infty}(\mathbb{R}^d)$ are defined as the set of (equivalence classes of) measurable functions $f:\mathbb{R}^d\to\mathbb{C}$ such that \eqref{eq:norm_weakLp} holds. %We use the convention $L^{\infty,\infty}:=L^\infty$. 

More generally, for $1\le p<\infty$ and $1\le q\leq\infty$, the Lorentz spaces $L^{p,q}=L^{p,q}(\mathbb{R}^d)$ are defined as the set of (equivalence classes of) measurable functions $f:\mathbb{R}^d\to\mathbb{C}$ such that the quasi-norm
%\begin{equation*}
%\|f\|_{L^{p,q}}:=p^{\frac1q}\Big( \int_0^\infty t^q \lambda \big( \{|f|>t\}\big)^{\frac{q}{p}} \, \frac{\diff t}{t} \Big)^{\frac1q} 
%\end{equation*}
\begin{equation*}
\|f\|_{L^{p,q}}:= p^{1/q} \| \lambda (  \{|f|>t\})^{1/p}\,t\|_{L^q ((0,\infty),\diff t / t)}  %0<p<\infty, 0<q\leq \infty,
\end{equation*}
is finite.

For $1\le p<\infty$ and $1\le q_1\leq q_2 \le \infty$, the continuous embedding $
L^{p,q_1}\subseteq L^{p,q_2}$ holds.
Moreover $L^{p,p}$ identifies with $L^p$. We will use the following generalizations of H\"older and Young's inequality in Lorentz spaces, see \cite{ONeil63,Yap69,Lemarie-Rieusset02,BezLeeNakamuraSawano17} or \cite[1.4.19]{Grafakos08}. %\cite[Prop.~2.3, 2.4]{LemarieRieussetRecentDeveloppementsOnTheNavierStokesProblem} Finalement LemarieRieusset ne me semble pas adapte, il y a des hypotheses non naturelles sur les $p_j,q_j$.

For $1\leq p_1,p_2<\infty$, $1\le q_1,q_2\le\infty$, H\"older's inequality states that
\begin{equation}\label{eq:weak_Holder}
\|f_1f_2\|_{L^{p,q}}\lesssim\|f_1\|_{L^{p_1,q_1}}\|f_2\|_{L^{p_2,q_2}},\qquad \frac{1}{p}=\frac{1}{p_1}+\frac{1}{p_2}\,, \quad \frac{1}{q}=\frac{1}{q_1}+\frac{1}{q_2}\,,
\end{equation}
whenever the right hand side is finite.

Young's inequality states that, for $1<p,p_1,p_2<\infty$, $1\le q_1,q_2\le\infty$, 
\begin{equation}\label{eq:weak_Young}
\|f_1*f_2\|_{L^{p,q}}\lesssim\|f_1\|_{L^{p_1,q_1}}\|f_2\|_{L^{p_2,q_2}}\,, \qquad 1+\frac{1}{p}=\frac{1}{p_1}+\frac{1}{p_2}\,,\quad \frac{1}{q}=\frac{1}{q_1}+\frac{1}{q_2}\,,
\end{equation}
and for $1< p<\infty$, $1\le q\le\infty$,
\begin{equation}\label{eq:weak_Young2}
\|f_1*f_2\|_{L^\infty}\lesssim\|f_1\|_{L^{p,q}}\|f_2\|_{L^{p',q'}} \,, \qquad \frac{1}{p}+\frac{1}{p'}=1\,,\quad \frac{1}{q}+\frac{1}{q'}=1\,.
\end{equation}

We have the following estimates that are used several times in Section \ref{sec:linear}. The first one is an obvious application of the usual H\"older and Young inequalities. The second and third ones are close to  the Hardy-Littlewood-Sobolev inequality but cannot be directly deduced from it. Recall the convention \eqref{eq:convfourier} on the Fourier transform.

\begin{lemma}\label{lm:convol-ineq}$ $
\begin{enumerate}[label=(\roman*)]
\item Let $u_1,u_2\in L^2$ and $W\in L^1$. Then, 
\begin{equation}\label{eq:convol-ineq1}
\big\| \bar{\FF}(W)*(u_1u_2) \big \|_{L^\infty}\lesssim \|W\|_{L^1}\|u_1\|_{L^2}\|u_2\|_{L^2}.
\end{equation}
\item Let $u_1,u_2\in \dot{H}^1$ and $W\in L^{3,\infty}$. Then $W\mathcal{F}(u_1u_2) \in L^1$ and 
\begin{align}\label{eq:convol-ineq2}
&\big\| \bar{\FF}(W)*(u_1u_2) \big \|_{L^\infty} \lesssim \|W\|_{L^{3,\infty}} \|u_1\|_{\dot{H}^1} \|u_2\|_{\dot{H}^1} .
\end{align}
\item Let $u_1\in L^2$, $u_2,u_3\in \dot{H}^1$ and $W\in L^{3,\infty}$. Then $W\mathcal{F}(u_1u_2) \in L^{3/2,\infty}$ and 
\begin{align}\label{eq:convol-ineq2_0}
&\big\| \big(\bar{\FF}(W)*(u_1u_2)\big)u_3 \big \|_{L^2} \lesssim \|W\|_{L^{3,\infty}} \|u_1\|_{L^2} \|u_2\|_{\dot{H}^1}\|u_3\|_{\dot{H}^1} .
\end{align}
\end{enumerate}
\end{lemma}
\begin{proof}
To simplify formulas below, we write $w=\bar{\FF}(W)$.

(i) directly follows from the H\"older and Young inequalities:
\begin{align*}
\big\| w*(u_1u_2) \big \|_{L^\infty} 
\lesssim \| w \|_{L^\infty} \big \| u_1u_2 \big \|_{L^1} \lesssim \| w \|_{L^\infty} \| u_1\|_{L^2}\|u_2\|_{L^2} ,
%&\lesssim \|W\|_{L^1} \|u_1\|_{L^2}^2\|u_2\|_{L^2}^2.
\end{align*}
which yields~\eqref{eq:convol-ineq1} using $\| w\|_{L^\infty}  \lesssim \| W\|_{L^1} $.

To prove (ii), we use first the  H\"older inequality in Lorentz spaces~\eqref{eq:weak_Holder},%--\eqref{eq:weak_Holder2},
\begin{align*}
\big\| w*(u_1u_2) \big\|_{L^\infty} \lesssim \big\| W\,\mathcal{F}(u_1u_2) \big\|_{L^1}
\lesssim \|W\|_{L^{3,\infty}}  \big\| \mathcal{F}(u_1u_2) \big\|_{L^{3/2,1}} \,,
\end{align*}
and then both the Young \eqref{eq:weak_Young} %--\eqref{eq:weak_Young2} 
and H\"older~\eqref{eq:weak_Holder} %--\eqref{eq:weak_Holder2}  
inequalities in Lorentz spaces
\begin{multline*}
\big\| \mathcal{F}(u_1u_2) \big\|_{L^{3/2,1}} 
\lesssim  \big\| \mathcal{F}(u_1)*\mathcal{F}(u_2) \big\|_{L^{3/2,1}} 
\lesssim   \big\| \mathcal{F}(u_1)\big\|_{L^{6/5,2}} \big\| \mathcal{F}(u_2) \big\|_{L^{6/5,2}}  \\
\lesssim  \big\| |k|^{-1} \big\|^2_{L^{3,\infty}} \big\| |k| \mathcal{F}(u_1)\big\|_{L^{2,2}} \big\| |k| \mathcal{F}(u_2) \big\|_{L^{2,2}}  
\lesssim  \|u_1\|_{\dot{H}^1} \|u_2\|_{\dot{H}^1} \, . %\label{eq:funct_ineq_lor}
\end{multline*}

To prove (iii), we use first the  Young inequality in Lorentz spaces~\eqref{eq:weak_Young},%--\eqref{eq:weak_Young2},
\begin{align*}
\big\| (w*(u_1u_2))u_3\big\|_{L^2}\lesssim\big\| (W\FF(u_1u_2))*\FF(u_3)\big \|_{L^2}
\lesssim \big\| W\FF(u_1u_2)\big\|_{L^{3/2,\infty}} \big\| \FF(u_3)\big \|_{L^{6/5,2}}\,.
\end{align*}
The term with $u_3$ is controlled with the  H\"older inequality in Lorentz spaces~\eqref{eq:weak_Holder},%--\eqref{eq:weak_Holder2},
\begin{align*}
\big\| \FF(u_3)\big \|_{L^{6/5,2}}
\lesssim  \||k|^{-1}\|_{L^{3,\infty}}\big\| |k|\FF(u_3)\big\|_{L^{2,2}}
\lesssim \|u_3\|_{\dot{H}^1} .
\end{align*}
The term with $u_1$ and $u_2$ is estimated using the  H\"older inequality in Lorentz spaces~\eqref{eq:weak_Holder} %--\eqref{eq:weak_Holder2} 
first, followed by the  Young inequality in Lorentz spaces~\eqref{eq:weak_Young},%--\eqref{eq:weak_Young2},
\begin{align*}
 \big\| W\FF(u_1u_2)\big\|_{L^{3/2,\infty}} 
\lesssim \|W\|_{L^{3,\infty}} \big\|\FF(u_1)*\FF(u_2)\big\|_{L^{3,\infty}} 
\lesssim \|W\|_{L^{3,\infty}} \big\|\FF(u_1)\big\|_{L^{2,\infty}} \big\|\FF(u_2)\big\|_{L^{6/5,\infty}}\,.
\end{align*}
The term
$ \|\FF(u_2)\|_{L^{6/5,\infty}} \lesssim \|\FF(u_2)\|_{L^{6/5,2}}$ is estimated in the same way as $ \|\FF(u_3)\|_{L^{6/5,2}} $, while
\begin{align*}\big\|\FF(u_1)\big\|_{L^{2,\infty}} 
\lesssim \big\|\FF(u_1)\big\|_{L^{2,2}}
\lesssim  \|u_1\|_{L^{2}} \, .
\end{align*}
This proves the lemma.
\end{proof}

\section{Proof of the main results}\label{sec:linear}

In this section, we prove the results stated in Section \ref{subsec:linear}. We begin with reducing the problem of the minimization of the Klein--Gordon--Schr\"odinger energy functional to the problem of the minimization of a well-chosen Hartree functional in Section \ref{sec:KGS-Hartree}. We prove the existence and uniqueness of a minimizer in Sections \ref{sec:exist-min-Hartree} and \ref{sec:uniq-min-Hartree}, respectively. In Section \ref{sec:asympt-Hartree}, we derive the asymptotic expansion of the ground state energy at small coupling. Finally, in Section~\ref{subsec:UV} we prove the convergence of the ground state and ground state energy in the ultraviolet limit.

Throughout this section, $C_V$ stands for a positive constant depending on $V$ which may vary from line to line.

\subsection{The Hartree energy functional}\label{sec:KGS-Hartree}

Recall that the Klein--Gordon--Schr\"odinger energy functional $\mathcal{E}(u,f)$ has been defined in~\eqref{eq:cE2}. The next proposition shows that the minimization problem for $\mathcal{E}$ reduces to the minimization problem for the Hartree energy.

\begin{proposition}\label{prop:KGS-Hartree}
Suppose that $V$ satisfies Hypothesis \ref{condV} and $W$ satisfies Hypothesis \ref{condv0:2a}. We have%for~$u\in\mathcal{U}$ and $f\in\mathcal{Z}_\omega$,
\begin{align}\label{eq:E-J-1}
\inf_{u\in\mathcal{U},f\in\mathcal{Z}_\omega} \mathcal{E}(u,f)=\inf_{u\in\mathcal{U}} J(u),
\end{align}
where 
\begin{align}
J(u)&:= \ps{u}{H_V\,u}-\big\|W^{\frac12}\FF(|u|^2)\big\|_{L^2}^2 \notag\\
&=\ps{u}{H_V\,u}-\int_{\R^3}\big(\bar{\FF}\left(W\right)*\abs{u}^2\big)(x) \, \abs{u(x)}^2\diff x = \mathcal{E}(u,f_u) \, ,\label{eq:Hartree_en}
\end{align}
 and
 \begin{equation}
 f_u:=-g\omega^{-1}v\FF(|u|^2)\,. \label{eq:f_u}
 \end{equation}
Moreover $u\mapsto (u,f_u)$ is a bijection between the minimizers of $J$ and those of~$\mathcal{E}$. 
\end{proposition}
\begin{proof}
The Klein--Gordon--Schr\"odinger energy functional can be written under the form
\begin{align}
 \mathcal{E}(u,f)&=\langle u , H_Vu\rangle + \big\langle\omega^{\frac12}f,\omega^{\frac12}f\big\rangle+2g\Real\big\langle\omega^{-\frac12}v\FF(|u|^2),\omega^{\frac12}f\big\rangle\notag\\
&=\langle u , H_Vu\rangle+\big\|\omega^{\frac12}f+g\omega^{-\frac12}v\FF(|u|^2)\big\|_{L^2}^2-g^2\big\|\omega^{-\frac12}v\FF(|u|^2)\big\|_{L^2}^2. \label{eq:cE2_2}
\end{align}
Note that $\omega^{1/2}f$ belongs to~$L^2$ since $f$ is in~$\mathcal{Z}_\omega$. Moreover, $\omega^{-\frac12}v\FF(|u|^2)$ belongs to~$L^2$ since
\begin{align*}
g^2\big\|\omega^{-\frac12}v\FF(|u|^2)\big\|_{L^2}^2&=\big\|W\big|\FF(|u|^2)\big|^2\big\|_{L^1} \\
&\le\|W_1\|_{L^1}\|\FF(|u|^2)\|^2_{L^\infty}+\|W_2\FF(|u|^2)\|_{L^1}\|\FF(|u|^2)\|_{L^\infty}\\
&\lesssim\|W_1\|_{L^1}\|u\|^4_{L^2}+\|W_2\|_{L^{3,\infty}}  \|u\|_{\dot{H}^1}^2 \|u\|_{L^2}^2<\infty,
\end{align*}
where we used Lemma \ref{lm:convol-ineq} in the second inequality. Eq.~\eqref{eq:cE2_2} then implies that 
\begin{equation*}
\mathcal{E}(u,f_u)=\langle u , H_Vu\rangle-g^2\big\|\omega^{-\frac12}v\FF(|u|^2)\big\|_{L^2}^2
=\min_{f\in \mathcal{Z}_\omega}\mathcal{E}(u,f) \,,
\end{equation*}
and $f_u$ is the unique minimizer of $\mathcal{E}(u,f)$ at fixed $u$.
By our convention \eqref{eq:convfourier} on the Fourier transform, the energy~$\mathcal{E}(u,f_u)$ can be rewritten as the Hartree energy~\eqref{eq:Hartree_en}, which yields the result.
%The strategy consists in applying the Lagrange Multiplier rule, so let us define the following Lagrangian on $\mathcal{Q}(H_{el})\times \ZZ_{\omega}\times\R$ :
%\begin{align*}
%\LL(u,f,\lambda)=\E(u,f)-\lambda\,\CC(u,f)
%\end{align*}
%with the constraint, defined on $\mathcal{Q}(H_{el})\times \ZZ_{\omega}$ :
%\begin{align*}
%\CC(u,f)= \norm{u}_{L^2}^2-1
%\end{align*}
%and where $\lambda\in\R$ is a Lagrange multiplier. We are thus lead to solve the following system of equations :
%\begin{align*}
%\partial_u\LL(u,f,\mu)&=\partial_u\E(u,f)-\lambda\partial_u\CC(u,f)=0\quad\text{in}\;\mathcal{Q}(H_V)^*,\\
%\partial_f\LL(u,f,\mu)&=\partial_f\E(u,f)-\lambda\,\partial_f\CC(u,f)=0,\quad\text{in}\;\ZZ_\omega^*,\\
%\partial_\mu\LL(u,f,\mu)&=\CC(u,f)=0\quad\text{in}\;\R.
%\end{align*}
%This yields, where the derivatives $\partial_u\E$ and $\partial_f\E$ are taken in the sense of Frechet :
%\begin{align*}
%\partial_u\E(u,f)&=H_V\,u+2g\displaystyle{\int_{\R^3}} e^{ikx}v(k)f(k)dk\;u\\
%\partial_f\E(u,f)&=\omega f+g\,v\ps{e^{ikx}u}{u}_{L^2}
%\end{align*}
%we must then solve :
%\begin{align}
%H_V\,u+2g\displaystyle{\int_{\R^3}} e^{ikx}v(k)f(k)dk\;u-\mu\,u&=0,\\
%\omega f+g\,v\ps{e^{ikx}u}{u}_{L^2}&=0,\label{null derivative : f}
%\end{align}
%Eq.~\eqref{null derivative : f} gives us an explicit form for $f$ in terms of $u:$
%\begin{equation}
%f=-g\,\omega^{-1}v\ps{e^{ikx}u}{u}_{L^2},
%\end{equation}
%which we can insert in the expression of the energy and obtain the functional 
\end{proof}
%We are thus lead to proving the existence and uniqueness of a minimizer for \eqref{eq:Hartree_en}.

%\subsection{Proof of Proposition \ref{prop:condfL2}}\label{sec:equiv-Hartree}
We now establish Proposition \ref{prop:condfL2} which gives necessary conditions and sufficient conditions so that $f_{\mathrm{gs}}$ belongs to~$L^2(\mathbb{R}^3)$, where $(u_{\mathrm{gs}},f_{\mathrm{gs}})$ a minimizer of the Klein--Gordon--Schr\"odinger energy functional $\mathcal{E}(u,f)$.

\begin{proof}[Proof of Proposition \ref{prop:condfL2}]
By Proposition \ref{prop:KGS-Hartree}, any minimizer $(u_{\mathrm{gs}},f_{\mathrm{gs}})$ of  the Klein--Gordon--Schr\"odinger energy functional over $\mathcal{U}\times\mathcal{Z}_\omega$ satisfies the relation~\eqref{eq:f_u}:
\begin{equation}
\forall k \in \R^3 , \quad f_{\mathrm{gs}}(k)=-g\omega^{-1}(k)v(k) \mathcal{F}(|u_{\mathrm{gs}}|^2)(k).
\end{equation}
Since $u_{\mathrm{gs}}$ is in~$H^1(\R^3)$, we have, for all $k$ in~$\mathbb{R}^3$,
\begin{equation*}
\big|\mathcal{F}(|u_{\mathrm{gs}}|^2)(k)\big|\lesssim \min(1,|k|^{-1}).
\end{equation*}
Moreover, since $\mathcal{F}(|u_{\mathrm{gs}}|^2)$ is continuous at $k=0$ and $\mathcal{F}(|u_{\mathrm{gs}}|^2)(0)=1$, we have 
\begin{equation*}
\frac12\le\big|\mathcal{F}(|u_{\mathrm{gs}}|^2)(k)\big|\le\frac32,
\end{equation*}
in a neighborhood of $k=0$.
Hence we see that sufficient conditions ensuring that $f_{\mathrm{gs}}$ belongs to~$L^2(\R^3)$ are
\begin{equation*}
k\mapsto \frac{v(k)}{|k|\omega(k)} \mathds{1}_{|k|\ge1} \in L^2(\R^3) \quad\text{and}\quad k\mapsto \frac{v(k)}{\omega(k)} \mathds{1}_{|k|\le1} \in L^2(\R^3),
\end{equation*}
while a necessary condition is
\begin{equation*}
 k\mapsto \frac{v(k)}{\omega(k)} \mathds{1}_{|k|\le1} \in L^2(\R^3).
\end{equation*}
This establishes Proposition \ref{prop:condfL2}.
\end{proof}

In the remainder of this section, we study the Hartree energy functional \eqref{eq:Hartree_en}. By Proposition \ref{prop:KGS-Hartree}, the results that we prove for the Hartree energy directly imply the corresponding results for the Klein--Gordon--Schr\"odinger energy.

\subsection{Existence of a minimizer}\label{sec:exist-min-Hartree}

The next  proposition gives the existence of a minimizer for the Hartree energy \eqref{eq:Hartree_en}. It implies Theorem \ref{thm:mainlnc} from the introduction.

\begin{proposition} \label{prop:min_Hartree}
Suppose that $V$ satisfies Hypothesis \ref{condV} and $W$ satisfies Hypothesis \ref{condv0:2a}. There exists $C_V>0$ such that, if the decompositions of $V$ and $W$ of the form $V=V_1+V_2$, $W=W_1+W_2$ in Hypotheses \ref{condV} and \ref{condv0:2a}, respectively, can be chosen such that
\begin{equation}\label{eq:condv:0_Hartree}
\|W_1\|_{L^1}+C_V\|W_2\|_{L^{3,\infty}}\le \delta(\mu_{V_1}-\mu_V)
\end{equation}
and
\begin{equation}\label{eq:smallness_Hartree}
C\|W_2\|_{L^{3,\infty}}\le \frac12(1-a),
\end{equation}
then the Hartree energy functional \eqref{eq:Hartree_en} has a minimizer. Here $C$ and $\delta$ are universal constants and $a$ is given by Hypothesis \ref{condV}.
\end{proposition}
As mentioned in the introduction, the existence of a minimizer for the Hartree energy has been proven under various conditions in different contexts, but we are not aware of a reference giving the result under our assumptions. We  detail the proof of Proposition \ref{prop:min_Hartree} in Appendix \ref{app:Hartree}.

\begin{remark}
 Writing the Hartree energy in its usual form
\begin{equation}%\label{eq:Hartree_en}
J(u)=\ps{u}{H_V\,u}-\int_{\R^3}(w*\abs{u}^2)(x)\, \abs{u(x)}^2\diff x,
\end{equation}
we have in our context $w=\bar\FF(W)$ (in the sense of distributions) and it is thus natural to make an assumption on the Fourier transform of the convolution potential $w$. In other contexts, however, it might be more natural to make hypotheses on $w$ rather than on its Fourier transform. It is  straightforward to verify that our proof adapts to the conditions
\begin{align*}
&w=w_1+w_2\in L^\infty(\mathbb{R}^3) + L^{3/2,\infty}(\mathbb{R}^3), \\
& \|w_1\|_{L^\infty}+C_V\|w_2\|_{L^{3/2,\infty}}\le \delta(\mu_{V_1}-\mu_V), \quad C\|w_2\|_{L^{3/2,\infty}}\le \frac12(1-a).
\end{align*}
It suffices to use, instead of \eqref{eq:convol-ineq1}--\eqref{eq:convol-ineq2}, the inequalities 
\begin{align*}
&\big\| (w*|u_1|^2)|u_2|^2\big \|_{L^1}\lesssim \|w\|_{L^\infty}\|u_1\|_{L^2}^2\|u_2\|_{L^2}^2\,,\\
&\big\| (w*|u_1|^2)|u_2|^2\big \|_{L^1}\lesssim \|w\|_{L^{3/2,\infty}}\|u_1\|_{\dot{H}^1}^2\|u_2\|_{L^2}^2\,,
\end{align*}
that can be proven using the weak H\"older and Young inequalities \eqref{eq:weak_Holder}--\eqref{eq:weak_Young2}.
\end{remark}
\begin{remark}
 As in Remark \ref{rk:intro_Nelson2}, in the case of a confining external potential $V$, one can always find a decomposition $V=V_1+V_2$ such that \eqref{eq:condv:0_Hartree} 
 %\eqref{prop:min_Hartree} 
 holds, by Lemma \ref{lm:confining_intro}. Hence a minimizer exists in this case without any restriction on the size of $\|W_1\|_{L^1}$. %More precisely, assuming that $v$ satisfies Hypothesis \ref{condv0} and $V=V_+-V_-$ is such that $V_+\in L^2_{\mathrm{loc}}(\mathbb{R}^3)$, $V_-\in L^{3/2}_{\mathrm{loc}}(\mathbb{R}^3)$, $V(x)\to\infty$ as $|x|\to\infty$, there exists $g_0>0$ such that if $\|W_2\|_{L^{3,\infty}}\le g_0$, then there exists a minimizer $u_{\mathrm{gs}}$ for $J_V$.
\end{remark}

\subsection{Uniqueness of the minimizer}\label{sec:uniq-min-Hartree}

Now that we have the existence of a minimizer for the Hartree energy functional~$J$, we prove the uniqueness of the minimizer. The next proposition implies Theorem \ref{thm:mainlnc2} from the introduction.
\begin{proposition}\label{prop:uniqueness Nelson}
Suppose that $V$ satisfies Hypotheses \ref{condV} and \ref{condGS} and that $W$ satisfies Hypothesis \ref{condv0:2a}. There exists $\varepsilon_V>0$ such that, if
\begin{equation*}
\|W\|_{L^1+L^{3,\infty}}\le\varepsilon_V,
\end{equation*}
then $J$ has a unique minimizer $u_{\mathrm{gs}}$ in~$\mathcal{U}$ such that $\langle u_{\mathrm{gs}},u_V\rangle_{L^2}>0$.
\end{proposition}
To prove Proposition \ref{prop:uniqueness Nelson} we use the following decomposition: $$L^2(\R^3)=\C u_V\oplus (\C u_V)^\perp,$$ where $u_V$ is the ground state of $H_V$ as in Hypothesis \ref{condGS}, and we write $u=\alpha u_V+\varphi,$ with the normalization condition~$\abs{\alpha}^2+\norm{\varphi}_{L^2}^2=1.$ The Hartree ground state energy is denoted by
\begin{equation*}
E_V:=\inf_{u\in\mathcal{U}}J(u).
\end{equation*}
We display the dependence on $V$ since the existence of a gap, $E_V<E_{V_1}$, is an important step in our proof of the existence of a minimizer in Appendix \ref{app:Hartree}.

Our proof of Proposition \ref{prop:uniqueness Nelson} relies on the following two lemmata. For $\lambda\in\mathbb{R}$, we set the resolvent~$R_\lambda:=(H_{V}-\lambda)^{-1}$ (defined, a priori, as an unbounded operator on $\mathrm{Ran}(\mathbf{1}_{\{\lambda\}}(H_V))^\perp$). We also recall that $\Pi_V^\perp:=\mathbf{I}-\ket{u_V}\bra{u_V}.$
\begin{lemma}\label{lm:phi}
Suppose that $V$ satisfies Hypotheses \ref{condV} and \ref{condGS} and that $W$ satisfies Hypothesis~\ref{condv0:2a}. There exists $\varepsilon_V>0$ such that, if
\begin{equation*}
\|W\|_{L^1+L^{3,\infty}}\le\varepsilon_V,
\end{equation*} 
then for all global minimizer $u$ in $\mathcal{U}$ of $J$ such that $u = \alpha u_V + \varphi$, with~$\alpha$ in~$\mathbb{C}$, $\varphi$ in~$\mathcal{Q}_V$, and~$u_V\perp \varphi$ in~$L^2$, we have
\begin{align}\label{eq:minimizer_phi}
%\alpha &= \left(1-\norm{\varphi}_{L^2}^2\right)^{1/2} \\
\varphi=2R_{\lambda_V}\Pi_V^\perp\,(\bar{\FF}(W)*\abs{u}^2)u,\quad \lambda_V:=E_V-\big\langle u,(\bar{\FF}(W)*|u|^2)u\big\rangle.
\end{align}
\end{lemma}
\begin{proof}
We first recall that if $u$ is a minimizer of $J$, then $u$ is a non-linear Hartree eigenstate, 
\begin{equation}\label{eq:Hartree_eigenstate}
\big(H_V - 2\bar{\FF}(W)*|u|^2\big)u=\lambda_V u\quad\text{in}\; \mathcal{Q}_V^*,
\end{equation}
where $\mathcal{Q}_V^*$ is the topological dual of $\mathcal{Q}_V$ and $\lambda_V$ is defined by \eqref{eq:minimizer_phi}.
 Here $H_V$ should be understood as an operator in $\LL(\mathcal{Q}_V,\mathcal{Q}_V^*).$ Eq.~\eqref{eq:Hartree_eigenstate} can be proven by using that, for all~$t$ in~$\mathbb{C}$ and all $u^\perp\in\mathcal{U}$ such that $u\perp u^\perp$ in $L^2$, we have
 \begin{align*}
 J\big((1+|t|^2)^{-\frac12}(u+tu^\perp)\big)\ge E_V.
 \end{align*}
Computing the term of order $1$ in~$t$ in the asymptotic expansion of the previous expression as~$t\to0$ shows that
\begin{equation*}
\big\langle u^\perp,\big(H_V - 2\bar{\FF}(W)*|u|^2\big)u\big\rangle=0,
\end{equation*}
 for all $u^\perp\in\mathcal{U}$ such that $u\perp u^\perp$ in $L^2$. Hence \eqref{eq:Hartree_eigenstate} holds for some $\lambda_V\in\mathbb{C}$. Since $J(u)=E_V$, we obtain that $\lambda_V$ is given by \eqref{eq:minimizer_phi}.

Now, applying $\Pi_V^\perp$ to \eqref{eq:Hartree_eigenstate} gives
\begin{equation}\label{eq:piperpphi}
(H_V - \lambda_V)\Pi_V^\perp u= 2\Pi_V^\perp\big(\bar{\FF}(W)*|u|^2\big)u.
\end{equation}
Let $\delta_V:=\mathrm{dist}(\mu_V,\sigma(H_V)\setminus\{\mu_V\})$ be the distance from $\mu_V$ to the rest of the spectrum of~$H_V$. Recall that Hypothesis~\ref{condGS} implies that~$\delta_V>0$. Using perturbative arguments, it is not difficult to verify that 
\begin{equation}\label{eq:EV<muV}
E_V \le \mu_V + \frac12\delta_V,
\end{equation}
provided that $\|W\|_{L^1+L^{3,\infty}}\le\varepsilon_V$, for some $\varepsilon_V>0$ small enough (see \eqref{eq:r1} in Appendix \ref{app:Hartree} for a precise justification). Moreover, since
\begin{equation*}
\big\langle u,(\bar{\FF}(W)*|u|^2)u\big\rangle=\big\|W^{\frac12}\FF(|u|^2)\big\|_{L^2}^2\ge0\,,
\end{equation*}
we also have $\lambda_V \leq E_V$, and hence
\begin{equation}\label{eq:lambda<}
\lambda_V \le \mu_V + \frac12\delta_V \,.
\end{equation}
 Therefore the operator 
 \begin{equation*}
 (H_V - \lambda_V)\Pi_V^\perp\in\LL(\mathcal{Q}_V,\mathcal{Q}_V^*)
 \end{equation*}
 is bounded invertible with inverse $R_{\lambda_V}\Pi_V^\perp\in\LL(\mathcal{Q}_V^*,\mathcal{Q}_V).$ Since $\varphi=\Pi_V^\perp u$, \eqref{eq:piperpphi} then implies~\eqref{eq:minimizer_phi}.
\end{proof}
In the following lemma, given $u_1$, $u_2$ two minimizers of $J$, we decompose as before, for $j=1,2$, $u_j = \alpha_j u_V + \varphi_j$, with~$\alpha_j=\langle u_j,u_V\rangle>0$, $\varphi_j\in\mathcal{Q}_V$ and $u_V\perp \varphi_j$ in~$L^2$.
\begin{lemma}\label{lm:phi_stab}
Suppose that $V$ satisfies Hypotheses~\ref{condV} and \ref{condGS} and that $W$ satisfies Hypothesis~\ref{condv0:2a}. There exists $\varepsilon_V>0$ such that, if
\begin{equation*}
\|W\|_{L^1+L^{3,\infty}}\le\varepsilon_V,
\end{equation*} 
then for all global minimizers $u_1$, $u_2$ in $\mathcal{U}$ of $J$, we have
\begin{align}
\|\varphi_1-\varphi_2\|_{\mathcal{Q}_V} \le C_V\|W\|_{L^1+L^{3,\infty}}  \big(\|u_1\|_{{H}^1}^2+\|u_2\|_{{H}^1}^2\big)^2 \, \| u_1 - u_2\|_{L^2}. \label{eq:estim_phi1-phi2}
\end{align}
\end{lemma}
\begin{proof}
Let $\lambda_1$, $\lambda_2$ be defined as in Lemma \ref{lm:phi}, namely $\lambda_j=E_V-\langle u_j,(\bar{\FF}(W)*|u_j|^2)u_j\rangle$. By Lemma \ref{lm:phi} and the triangle inequality, we have
\begin{align*}
\|\varphi_1-\varphi_2\|_{\mathcal{Q}_V}&\le S_1+S_2,
\end{align*}
where
\begin{align*}
&S_1:=2\big\|\big(R_{\lambda_1}\Pi_V^\perp-R_{\lambda_2}\Pi_V^\perp\big)(\bar{\FF}(W)*\abs{u_1}^2)u_1\big\|_{\mathcal{Q}_V} ,\\
&S_2:=2\big\|R_{\lambda_2}\Pi_V^\perp\big((\bar{\FF}(W)*\abs{u_1}^2)u_1-(\bar{\FF}(W)*\abs{u_2}^2)u_2\big)\big\|_{\mathcal{Q}_V}.
\end{align*}

We first estimate the second term. From \eqref{eq:lambda<} we deduce that $R_{\lambda_2}\Pi_V^\perp\in\LL(\mathcal{Q}_V^*,\mathcal{Q}_V)$ and
\begin{equation}\label{eq:estim_R2}
\big\|R_{\lambda_2}\Pi_V^\perp\big\|_{\LL(\mathcal{Q}_V^*,\mathcal{Q}_V)}\le4\delta_V^{-1}.
\end{equation}
Hence we can estimate
\begin{align*}
S_2
&\le 8\delta_V^{-1}\norm{(\bar{\FF}(W)*\abs{u_1}^2)u_1-(\bar{\FF}(W)*\abs{u_2}^2)u_2}_{\mathcal{Q}_V^*}\\
&\le 8\delta_V^{-1}\norm{(\bar{\FF}(W)*\abs{u_1}^2)u_1-(\bar{\FF}(W)*\abs{u_2}^2)u_2}_{L^2},
\end{align*}
since $L^2\subset \mathcal{Q}_V^*$. We obtain from the triangle inequality that
\begin{align*}
S_2&\le 8\delta_V^{-1}\norm{\big(\bar{\FF}(W)*[(\bar u_1 -\bar u_2 ) u_1]\big )u_1}_{L^2}\\
&\quad+ 8\delta_V^{-1}\norm{\big(\bar{\FF}(W)*[\bar u_2( u_1 - u_2 )]\big )u_1}_{L^2}\\
&\quad+ 8\delta_V^{-1}\norm{\big(\bar{\FF}(W)*|u_2|^2\big )(u_1-u_2)}_{L^2} ,
\end{align*}
and hence Lemma \ref{lm:convol-ineq} yields
\begin{align}\label{eq:estim_S2_a}
S_2&\le C_V\|W\|_{L^1+L^{3,\infty}} \big(\|u_1\|_{{H}^1}^2+\|u_2\|_{{H}^1}^2\big) \| u_1 - u_2\|_{L^2}.
\end{align}

Now we estimate $S_1$. It follows from the resolvent equation that
\begin{equation*}
S_1\le2|\lambda_1-\lambda_2|\big\|\big(R_{\lambda_1}\Pi_V^\perp R_{\lambda_2}\Pi_V^\perp\big)(\bar{\FF}(W)*\abs{u_1}^2)u_1\big\|_{\mathcal{Q}_V}.
\end{equation*}
Using \eqref{eq:estim_R2} twice, first for $R_{\lambda_2}$ and next for $R_{\lambda_1}$ (using also that $\mathcal{Q}_V\subset \mathcal{Q}_V^*$ in the latter case), and then appying Lemma \ref{lm:convol-ineq}, we obtain
\begin{equation}\label{eq:estim_S1a_b}
S_1\le C_V\|W\|_{L^1+L^{3,\infty}} \|u_1\|_{{H}^1}^2 |\lambda_1-\lambda_2| .
\end{equation}
Since $\lambda_j=E_V-\langle u_j,(\FF(W)*|u_j|^2)u_j\rangle$ and $\|u_j\|_{L^2}=1$, we can estimate, by the triangle inequality,
\begin{align*}
|\lambda_1-\lambda_2|\le \|u_1-u_2\|_{L^2}\big\|(\bar{\FF}(W)*|u_1|^2)u_1\big\|_{L^2}+\norm{(\bar{\FF}(W)*\abs{u_1}^2)u_1-(\bar{\FF}(W)*\abs{u_2}^2)u_2}_{L^2}.
\end{align*}
Using Lemma \ref{lm:convol-ineq} to bound the first term, and the same argument we used to prove \eqref{eq:estim_S2_a} for the second one, we obtain
\begin{align}\label{eq:estim_S1b}
|\lambda_1-\lambda_2|\le C_V\|W\|_{L^1+L^{3,\infty}} \big(\|u_1\|_{{H}^1}^2+\|u_2\|_{{H}^1}^2\big) \| u_1 - u_2\|_{L^2}.
\end{align}
Putting together \eqref{eq:estim_S1a_b} and \eqref{eq:estim_S1b} yields
\begin{equation}\label{eq:estim_S1}
S_1\le C_V\|W\|_{L^1+L^{3,\infty}}^2 \|u_1\|_{{H}^1}^2 \big(\|u_1\|_{{H}^1}^2+\|u_2\|_{{H}^1}^2\big) \| u_1 - u_2\|_{L^2}.
\end{equation}

Using $\|W\|_{L^1+L^{3,\infty}}\le\varepsilon_V$ and combining \eqref{eq:estim_S1} and \eqref{eq:estim_S2_a}, we obtain \eqref{eq:estim_phi1-phi2}. Since the role of $\varphi_1$ and $\varphi_2$ in the left-hand-side of \eqref{eq:estim_phi1-phi2} is symmetric, this concludes the proof.
\end{proof}

We are now ready to prove the uniqueness of a minimizer for the Hartree energy \eqref{eq:Hartree_en}.

\begin{proof}[Proof of Proposition \ref{prop:uniqueness Nelson}]
By Proposition \ref{prop:min_Hartree}, we know that the Hartree functional energy $J$ has a minimizer for $\|W\|_{L^1+L^{3,\infty}}\le\varepsilon_V$ with $\varepsilon_V$ small enough. Let $u_1$, $u_2$ be two minimizers of~$J$. We use the same notations as in the previous proof, decomposing $u_j=\alpha_ju_V+\varphi_j$. Since $\alpha_j=\langle u_j,u_V\rangle>0$, in order to verify that $u_1=u_2$, it suffices to prove that $\varphi_1=\varphi_2$. 

To this end, we first show that $\|u_j\|_{\dot{H}^1}^2$ is bounded by $C_V$ and next apply Lemma \ref{lm:phi_stab}. We can write
\begin{align*}
\|u_j\|_{\dot{H}^1}^2 \le \langle u_j,H_{V_+}u_j\rangle = J(u_j) + \langle u_j , V_- u_j \rangle - \big\langle u_j,(\bar{\FF}(W)*|u_j|^2)u_j\big\rangle.
\end{align*}
We have $J(u_j)=E_V$ since $u_j$ is a minimizer, $\langle u_j,V_-u_j\rangle\le a\|u_j\|_{\dot{H}^1}^2+b$ with $a<1$ by Hypothesis \ref{condV} and $|\langle u_j,(\bar{\FF}(W)*|u_j|^2)u_j\rangle|\le C\|W\|_{L^1+L^{3,\infty}}\|u_j\|_{{H}^1}^2$ for some universal constant $C$ by Lemma \ref{lm:convol-ineq}. Therefore
\begin{align*}
\|u_j\|_{{H}^1}^2 \le 1+ \frac{1}{1-a-C\|W\|_{L^1+L^{3,\infty}}} \big ( E_V + b + C\|W\|_{L^1+L^{3,\infty}}  \big).
\end{align*}
Hence, since in addition, by \eqref{eq:EV<muV}, $E_V \le \mu_V + \frac12\delta_V$ for $\|W\|_{L^1+L^{3,\infty}}\le\varepsilon_V$ with $\varepsilon_V$ small enough, we deduce that $\|u_j\|_{{H}^1}^2\le C_V$. Lemma \ref{lm:phi_stab} then implies that
\begin{align}\label{eq:estim_phi1-phi2_d}
\|\varphi_1-\varphi_2\|_{\mathcal{Q}_V}&\le C_V\|W\|_{L^1+L^{3,\infty}} \| u_1 - u_2\|_{L^2}.
\end{align}

Now we have
\begin{align}
\| u_1 - u_2\|_{L^2}^2=|\alpha_1-\alpha_2|^2+\|\varphi_1-\varphi_2\|^2_{L^2}. \label{eq:estim_phi1-phi2_e}
\end{align}
To conclude the proof, we show that $|\alpha_1-\alpha_2|$ can be controlled by $\|\varphi_1-\varphi_2\|^2_{L^2}$. Lemma \ref{lm:phi} and the arguments used in the proof of Lemma \ref{lm:phi_stab} ensure that $\|\varphi_j\|_{L^2}\le\frac12$. Indeed,
\begin{align*}
\|\varphi_j\|_{L^2}\le\|\varphi_j\|_{\mathcal{Q}_V}&=2\big\|R_{\lambda_j}\Pi_V^\perp\,(\bar{\FF}(W)*\abs{u_j}^2)u_j\big\|_{\mathcal{Q}_V}\\
&\le C_V\big\|(\bar{\FF}(W)*\abs{u_j}^2)u_j\big\|_{\mathcal{Q}_V^*}\\
&\le C_V\big\|(\bar{\FF}(W)*\abs{u_j}^2)u_j\big\|_{L^2}\\
&\le C_V\|W\|_{L^1+L^{3,\infty}} \|u_j\|_{{H}^1}^2\\
&\le C_V\|W\|_{L^1+L^{3,\infty}} ,
\end{align*}
where in the last inequality we used in addition that $\|u_j\|_{\dot{H}^1}^2\le C_V$. Therefore $\|\varphi_j\|_{L^2}\le\frac12$ for $\|W\|_{L^1+L^{3,\infty}}\le\varepsilon_V$ with $\varepsilon_V$ small enough and we can thus estimate
\begin{align}
\abs{\alpha_1-\alpha_2}&=\abs{\left(1-\norm{\varphi_1}_{L^2}^2\right)^{1/2}-\left(1-\norm{\varphi_2}_{L^2}^2\right)^{1/2}}\notag\\
&=\abs{\frac{\norm{\varphi_2}_{L^2}+\norm{\varphi_1}_{L^2}}{\left(1-\norm{\varphi_1}_{L^2}^2\right)^{1/2}+\left(1-\norm{\varphi_2}_{L^2}^2\right)^{1/2}}}\abs{\norm{\varphi_2}_{L^2}-\norm{\varphi_1}_{L^2}}\notag\\
&\le C\norm{\varphi_1-\varphi_2}_{L^2}. \label{eq:control_alpha}
\end{align}
Inserting this into \eqref{eq:estim_phi1-phi2_d}--\eqref{eq:estim_phi1-phi2_e} and using that $\|\varphi_1-\varphi_2\|^2_{L^2}\le\|\varphi_1-\varphi_2\|^2_{\mathcal{Q}_V}$, we finally conclude that
\begin{equation*}
\|\varphi_1-\varphi_2\|_{\mathcal{Q}_V}\le C_V\|W\|_{L^1+L^{3,\infty}} \|\varphi_1-\varphi_2\|_{\mathcal{Q}_V}.
\end{equation*}
For $\|W\|_{L^1+L^{3,\infty}}\le\varepsilon_V$ with $\varepsilon_V$ small enough, this implies that $\varphi_1=\varphi_2$, which concludes the proof of the proposition.
\end{proof}

\subsection{Expansion of the ground state energy for small coupling constants}\label{sec:asympt-Hartree}

Assuming that $V$ satisfies Hypotheses \ref{condV} and \ref{condGS} and that $W=g^2\omega^{-1}v^2$ satisfies Hypothesis \ref{condv0:2a}, it follows from Proposition \ref{prop:uniqueness Nelson} that $J$ has a unique ground state in $\mathcal{U}$, denoted by  $u_{\mathrm{gs}}$, such that~$\langle u_{\mathrm{gs}},u_V\rangle_{L^2}>0$. The next proposition provides the asymptotic expansion of the Hartree ground state energy (or equivalently the Klein--Gordon--Schr\"odinger ground state energy) stated in Proposition \ref{thm:mainlnc3}.
\begin{proposition}\label{prop:asymptotics_Nelson}
Suppose that $V$ satisfies Hypotheses \ref{condV} and \ref{condGS} and that $W=g^2\omega^{-1}v^2$ satisfies Hypothesis \ref{condv0:2a}. Then, as $g\to0$,
\begin{equation*}
E_V=J(u_{\mathrm{gs}})=\mu_V-g^2\int \big( \bar{\FF}\left(\omega^{-1}v^2\right)*\abs{u_V}^2\big) (x)\abs{u_V(x)}^2\diff x+\mathcal{O}(g^4).
\end{equation*}
\end{proposition}
\begin{proof}
As in Lemma \ref{lm:phi}, we decompose the Hartree ground state $u_{\mathrm{gs}} = \alpha_{\mathrm{gs}} u_V + \varphi_{\mathrm{gs}}$, with~$\alpha_{\mathrm{gs}}$ in~$\mathbb{C}$, $\varphi_{\mathrm{gs}}$ in~$\mathcal{Q}_V$ and~$u_V\perp \varphi_{\mathrm{gs}}$ in~$L^2$. By Lemma \ref{lm:phi}, we have
\begin{align*}
\|\varphi_{\mathrm{gs}}\|_{\mathcal{Q}_V}&=2g^2\big\|R_{\lambda_V}\Pi_V^\perp\,(\bar{\FF}(\omega^{-1}v^2)*\abs{u_{\mathrm{gs}}}^2)u_{\mathrm{gs}}\big\|_{\mathcal{Q}_V} \\
&\le C_V g^2 \big \| (\bar{\FF}(\omega^{-1}v^2)*\abs{u_{\mathrm{gs}}}^2)u_{\mathrm{gs}}\big\|_{\mathcal{Q}_V^*} \\
&\le C_V g^2 \big \| (\bar{\FF}(\omega^{-1}v^2)*\abs{u_{\mathrm{gs}}}^2)u_{\mathrm{gs}}\big\|_{L^2} = \mathcal{O}(g^2),
\end{align*}
the last equality being a consequence of Lemma \ref{lm:convol-ineq}. Similarly, using that $R_{\lambda_V}=(H_V-\lambda_V)^{-1}$, we can estimate
\begin{align}
| \langle \varphi_{\mathrm{gs}},H_V\varphi_{\mathrm{gs}}\rangle | &\le |\lambda_V| \|\varphi_{\mathrm{gs}}\|_{L^2}^2
+g^4 \big|\big\langle \Pi_V^\perp\,(\bar{\FF}(\omega^{-1}v^2)*\abs{u_{\mathrm{gs}}}^2)u_{\mathrm{gs}} , R_{\lambda_V} \Pi_V^\perp\,(\bar{\FF}(\omega^{-1}v^2)*\abs{u_{\mathrm{gs}}}^2)u_{\mathrm{gs}}\big\rangle \big| \notag\\
&= \mathcal{O}(g^4). \label{eq:estima}
\end{align}
Here we used \eqref{eq:lambda<} which shows that $|\lambda_V|$ is bounded by $C_V$ and that $R_{\lambda_V}\Pi_V^\perp\in\LL(\mathcal{Q}_V^*,\mathcal{Q}_V).$ Besides, since $|\alpha_{\mathrm{gs}}|^2=1-\|\varphi_{\mathrm{gs}}\|_{L^2}^2$, we have
\begin{equation}\label{eq:estimb}
|\alpha_{\mathrm{gs}}|=1+\mathcal{O}(g^4).
\end{equation}
We can then compute
\begin{align*}
J(u_{\mathrm{gs}})&=\ps{\alpha_{\mathrm{gs}} u_V+\varphi_{\mathrm{gs}}}{H_V(\alpha_{\mathrm{gs}} u_V+\varphi_{\mathrm{gs}})}_{L^2}\\
&\quad-g^2\int \bar{\FF}(\omega^{-1}v^2)*\abs{\alpha_{\mathrm{gs}} u_V+\varphi_{\mathrm{gs}}}^2(x)\abs{\alpha_{\mathrm{gs}} u_V(x)+\varphi_{\mathrm{gs}}(x)}^2\diff x\\
&=\abs{\alpha_{\mathrm{gs}}}^2 \mu_V+\ps{\varphi_{\mathrm{gs}}}{H_V\varphi_{\mathrm{gs}}}_{L^2}\\
&\quad-g^2\abs{\alpha_{\mathrm{gs}}}^4\int \bar{\FF}(\omega^{-1}v^2)*\abs{u_V}^2(x)\abs{u_V(x)}^2\diff x + \mathcal{O}(g^4),
\end{align*}
where in the second equality we used that $H_Vu_V=\mu_Vu_V$, $u_V\perp \varphi_{\mathrm{gs}}$ in $L^2$, and Lemma \ref{lm:convol-ineq} in order to obtain the expansion of the convolution term.
Inserting \eqref{eq:estima} and \eqref{eq:estimb} into the last equality concludes the proof of the proposition.
\end{proof}

We conclude this section with the proof of Proposition \ref{prop:diff_GSE} which provides the difference, at second order in the coupling constant, between the ground state energy of the Pauli-Fierz Hamiltonian $\mathbb{H}$ and the ground state energy of the Klein--Gordon--Schr\"odinger energy functional.

\begin{proof}[Proof of Proposition \ref{prop:diff_GSE}]
Under the conditions of Proposition \ref{prop:diff_GSE}, using perturbative methods developed in the literature to study ground states of Pauli-Fierz Hamiltonians (see e.g. \cite{BachFrohlichSigal99,Hainzl03,BachFrohlichPizzo09,GriesemerHasler09,Sigal09,HaslerHerbst11}), it is not difficult to verify that the second-order asymptotic expansion of $\inf\sigma(\mathbb{H})$ is given by
\begin{multline*}
\inf \sigma(\mathbb{H})\\
=\mu_V-g^2 \ps{u_V\otimes\Omega}{a(h_x)(\Pi_V\otimes\Pi_\Omega)^\perp\big(\mathbb{H}_{\mathrm{free}}-\mu_V\big)^{-1} (\Pi_V\otimes\Pi_\Omega)^\perp a^*(h_x)u_V\otimes\Omega} +o(g^2) \,,
\end{multline*}
where we recall that $\Pi_V$ stands for the orthogonal projection onto the ground state $u_V$ of $H_V$, $\Pi_\Omega$ stands for the orthogonal projection onto the Fock vacuum $\Omega$, and $(\Pi_V\otimes\Pi_\Omega)^\perp=\mathbf{I}-\Pi_\Omega$. Moreover, $\mathbb{H}_{\mathrm{free}}=H_V\otimes\mathbf{I}_{\mathrm{f}}+\mathbf{I}_{\mathrm{el}}\otimes \mathbb{H}_{\mathrm{f}}$. Decomposing $(\Pi_V\otimes\Pi_\Omega)^\perp=\Pi_V\otimes\Pi_\Omega^\perp+\Pi_V^\perp\otimes\mathbf{I}_\text{f}$ and using \eqref{eq:comput-2ndorder}, we obtain
\begin{multline*}
\inf \sigma(\mathbb{H})=\mu_V-g^2 \int_{\R^3} \big(\bar{\FF}(\omega^{-1}v^2)*\abs{u_V}^2\big)(x)\abs{u_V(x)}^2\diff x\\
-g^2 \ps{u_V\otimes\Omega}{a(h_x)(\Pi_V^\perp\otimes\mathbf{I}_{\mathrm{f}})\big(\mathbb{H}_{\mathrm{free}}-\mu_V\big)^{-1}(\Pi_V^\perp\otimes\mathbf{I}_{\mathrm{f}})a^*(h_x)u_V\otimes\Omega} +o(g^2) \,.
\end{multline*}
A direct computation gives
\begin{multline*}
 \ps{u_V\otimes\Omega}{a(h_x)(\Pi_V^\perp\otimes\mathbf{I}_{\mathrm{f}})\big(\mathbb{H}_{\mathrm{free}}-\mu_V\big)^{-1}(\Pi_V^\perp\otimes\mathbf{I}_{\mathrm{f}})a^*(h_x)u_V\otimes\Omega} \\
=\int_{\mathbb{R}^3} v(k)^2 \big\langle u_V , e^{ikx} \Pi_V^\perp\big({H}_V-\mu_V+|k|\big)^{-1} \Pi_V^\perp e^{-ikx} u_V\big\rangle_{L^2_x}\mathrm{d}k, 
\end{multline*}
which, together with Proposition \ref{thm:mainlnc3}, proves Proposition \ref{prop:diff_GSE}.
\end{proof}

\subsection{Ultraviolet limit}\label{subsec:UV}

We suppose in this section that the coupling function is of the form~$v_\Lambda=v\mathds{1}_{|k|\le\Lambda}$ for some ultraviolet parameter $\Lambda>0$ and we write
\begin{equation*}
W_\Lambda:=g^2\omega^{-1}v_\Lambda=W\mathds{1}_{|k|\le\Lambda},
\end{equation*}
where we recall that $\omega$ stands for the dispersion relation for the bosons and $W=g^2\omega^{-1}v^2$.

Our fist concern is to show that the ground state energies $E_{V,\Lambda}$ defined by
\begin{equation*}
E_{V,\Lambda}:=\inf_{u\in\mathcal{U}}J_\Lambda(u),  \quad J_\Lambda(u):=\ps{u}{H_V\,u}-\int_{\R^3}\big(\bar{\FF}\left(W_\Lambda\right)*\abs{u}^2\big)(x)  \abs{u(x)}^2\diff x,
\end{equation*}
converge to $E_V$ as $\Lambda\to\infty$. As mentioned in the introduction, the main difficulty comes from the fact that, in general, $W_\Lambda$ does \emph{not} converge to $W$ in $L^1+L^{3,\infty}$. Nevertheless, we can rely on the following easy lemma.
\begin{lemma}\label{lm:Leb}
Suppose that $W$ satisfies Hypothesis \ref{condv0:2a}. Let $u$ in~$\mathcal{Q}_V$. Then
\begin{equation}\label{eq:lim1}
\int_{\R^3}\big(\bar{\FF}\left(W_\Lambda\right)*\abs{u}^2\big)(x)  \abs{u(x)}^2\diff x \underset{\Lambda\to\infty}{\longrightarrow} \int_{\R^3}\big(\bar{\FF}(W)*\abs{u}^2\big)(x)  \abs{u(x)}^2\diff x,
\end{equation}
and
\begin{equation}\label{eq:lim2}
\big\| (\bar{\FF}(W-W_\Lambda)*\abs{u}^2)u \big\|_{L^2} \underset{\Lambda\to\infty}{\longrightarrow} 0.
\end{equation}
\end{lemma}
\begin{proof}
To prove \eqref{eq:lim1}, we compute
\begin{align*}
&\int_{\R^3}\big(\bar{\FF}\left(W\right)*\abs{u}^2\big)(x)  \abs{u(x)}^2\diff x - \int_{\R^3}\big(\bar{\FF}(W_\Lambda)*\abs{u}^2\big)(x)  \abs{u(x)}^2\diff x \\
&=\int_{\R^3} (W(k)-W_\Lambda(k)) \big|\FF(|u|^2)\big|^2(k) \diff k =\int_{\R^3} W(k)\mathds{1}_{|k|\ge\Lambda} \big|\FF(|u|^2)\big|^2(k) \diff k.
\end{align*}
Clearly, for all $k$ in~$\mathbb{R}^3$, $W(k)\mathds{1}_{|k|\ge\Lambda} |\FF(|u|^2)|^2(k)\to 0$ as $\Lambda\to\infty$. Moreover, since $W$ is in~$L^1+L^{3,\infty}$ and $u$ in~$\mathcal{Q}_V$, the proof of Lemma \ref{lm:convol-ineq} shows that  $W\FF(|u|^2)$ is in~$ L^1$ and hence~$W |\FF(|u|^2)|^2$ belongs to~$ L^1$. Lebesgue's dominated convergence Theorem then proves~\eqref{eq:lim1}.

The proof of \eqref{eq:lim2} is similar, using Lemma \ref{lm:convol-ineq}(iii) together with Lebesgue's dominated convergence in $L^2$.
\end{proof}
Now we can prove the convergence of the ground state energies in the ultraviolet limit. The next proposition implies Proposition \ref{prop:conv-GS-en-intro} from the introduction.
\begin{proposition}[Ultraviolet limit of the ground state energies]\label{prop:conv-GS-en}
Suppose that $V$ satisfies Hypothesis \ref{condV} and that $W$ satisfies Hypothesis \ref{condv0:2a}. Then
\begin{equation*}
E_{V,\Lambda} \underset{\Lambda\to\infty}{\longrightarrow} E_{V}.
\end{equation*}
\end{proposition}
\begin{proof}
First, we observe that for $0<\Lambda\le\Lambda'$,
\begin{align*}
\int_{\R^3}\big(\bar{\FF}\left(W_\Lambda\right)*\abs{u}^2\big)(x) \abs{u(x)}^2\diff x &= \int_{\R^3} W_\Lambda(k) \big|\FF(|u|^2)\big|^2(k) \diff k \le \int_{\R^3} W_{\Lambda'}(k) \big|\FF(|u|^2)\big|^2(k) \diff k,
\end{align*}
and therefore $J_\Lambda\ge J_{\Lambda'}$. Hence $\Lambda\mapsto E_{V,\Lambda}$ is non-increasing on $(0,\infty)$ and bounded below by~$E_{V}$. We set 
\begin{equation*}
E_{V,\infty}:=\lim_{\Lambda\to\infty}E_{V,\Lambda} \ge E_{V}.
\end{equation*}

Now we show that $E_{V,\infty}\le E_{V}$. Let $\varepsilon>0$ and let $u_\varepsilon\in\mathcal{U}$ be such that $J(u_\varepsilon)\le E_{V}+\varepsilon$. We have
\begin{align*}
E_{V,\Lambda}&\le J_\Lambda(u_\varepsilon) = J(u_\varepsilon) - \int_{\R^3} (W_\Lambda(k)-W(k)) \big|\FF(|u_\varepsilon|^2)\big|^2(k) \diff k \\
&\le E_V+\varepsilon - \int_{\R^3} (W_\Lambda(k)-W(k)) \big|\FF(|u_\varepsilon|^2)\big|^2(k) \diff k.
\end{align*}
Applying Lemma \ref{lm:Leb}, this yields
\begin{equation*}
E_{V,\infty} = \lim_{\Lambda\to\infty} E_{V,\Lambda}\le E_{V}+\varepsilon.
\end{equation*}
Since $\varepsilon>0$ is arbitrary, this concludes the proof of the proposition.
\end{proof}

Next, we establish the convergence of the ground states of $J_\Lambda$ to the ground state of $J$, as~$\Lambda\to\infty$. Combined with Proposition \ref{prop:KGS-Hartree}, the next result implies Proposition \ref{prop:conv-GS2-en-intro} from the introduction. Some arguments of the proof below are similar to those used in the proof of the uniqueness of the minimizer of $J$ in Section \ref{sec:uniq-min-Hartree}. We do not give all the details.
\begin{proposition}[Ultraviolet limit of the ground states]
Suppose that $V$ satisfies Hypotheses \ref{condV} and \ref{condGS} and that $W$ satisfies Hypothesis \ref{condv0:2a}. There exists $\varepsilon_V>0$ such that, if
\begin{equation*}
\|W\|_{L^1+L^{3,\infty}}\le\varepsilon_V,
\end{equation*}
then for all $\Lambda>0$, $J_\Lambda$ and $J$ have unique minimizers $u_{\Lambda,\mathrm{gs}}$ and $u_{\mathrm{gs}}$ in $\mathcal{U}$, respectively, such that $\langle u_{\Lambda,\mathrm{gs}},u_V\rangle_{L^2}>0$ and $\langle u_{\mathrm{gs}},u_V\rangle_{L^2}>0$. They satisfy
\begin{equation*}
\big\|u_{\Lambda,\mathrm{gs}} - u_{\mathrm{gs}} \big\|_{\mathcal{Q}_V} \underset{\Lambda\to\infty}{\longrightarrow} 0.
\end{equation*}
\end{proposition}
\begin{proof}
Existence of a unique minimizer for $J$ follows from Proposition \ref{prop:uniqueness Nelson}. The same holds for $J_\Lambda$ since, for any $\Lambda>0$,
\begin{equation}\label{eq:Wlambda<Wlam}
\|W_\Lambda\|_{L^1+L^{3,\infty}}\le\|W\|_{L^1+L^{3,\infty}}.
\end{equation}
We write $u=u_{\mathrm{gs}}$ and $u_\Lambda=u_{\Lambda,\mathrm{gs}}$ to simplify expressions below. We decompose $u = \alpha u_V + \varphi$, with a coefficient~$\alpha=\langle u,u_V\rangle>0$, $\varphi$ in~$\mathcal{Q}_V$ and $u_V\perp \varphi$ in~$L^2$, and likewise $u_\Lambda = \alpha_\Lambda u_V + \varphi_\Lambda$, with a coefficient~$\alpha_\Lambda=\langle u_\Lambda,u_V\rangle>0$, $\varphi_\Lambda$ in~$\mathcal{Q}_V$ and $u_V\perp \varphi_\Lambda$ in~$L^2$. 

In the same way as in the proof of Proposition \ref{prop:uniqueness Nelson}, using also \eqref{eq:Wlambda<Wlam} and the fact that $E_{V,\Lambda}$ is uniformly bounded in $\Lambda$ (since $E_{V,\Lambda}$ converges as $\Lambda\to\infty$, by Proposition \ref{prop:conv-GS-en}), we have
\begin{align}\label{eq:boundphiLambda}
\|\varphi\|_{L^2}\le C_V\|W\|_{L^1+L^{3,\infty}} ,\quad \|\varphi_\Lambda\|_{L^2}\le C_V\|W\|_{L^1+L^{3,\infty}} .
\end{align}
This yields
\begin{align}
\| u - u_\Lambda\|_{\mathcal{Q}_V}\le C_V|\alpha-\alpha_\Lambda|+\|\varphi-\varphi_\Lambda\|_{\mathcal{Q}_V}. \label{eq:ult1}%\le C_V\|\varphi-\varphi_\Lambda\|_{\mathcal{Q}_V}, \label{eq:ult1}
\end{align}
The difference $|\alpha-\alpha_\Lambda|$ can be controlled by $\|\varphi-\varphi_\Lambda\|^2_{L^2}$. Indeed, we have the upper bound~$\|\varphi_\sharp\|_{L^2}\le C_V\|W\|_{L^1+L^{3,\infty}}\le\frac12$ for $\|W\|_{L^1+L^{3,\infty}}\le\varepsilon_V$ with $\varepsilon_V$ small enough, where~$\varphi_\sharp$ stands for $\varphi$ or $\varphi_\Lambda$. We can thus estimate in the same way as in \eqref{eq:control_alpha},
\begin{equation*}
\abs{\alpha-\alpha_\Lambda}=\abs{\frac{\norm{\varphi_\Lambda}_{L^2}+\norm{\varphi}_{L^2}}{\left(1-\norm{\varphi}_{L^2}^2\right)^{1/2}+\left(1-\norm{\varphi_\Lambda}_{L^2}^2\right)^{1/2}}}\abs{\norm{\varphi_\Lambda}_{L^2}-\norm{\varphi}_{L^2}}\le C\norm{\varphi-\varphi_\Lambda}_{L^2}.
\end{equation*}
Since $\norm{\varphi-\varphi_\Lambda}_{L^2}\le\norm{\varphi-\varphi_\Lambda}_{\mathcal{Q}_V}$, inserting the previous inequality into \eqref{eq:ult1} gives
\begin{align}
\| u - u_\Lambda\|_{\mathcal{Q}_V}\le C_V\|\varphi-\varphi_\Lambda\|_{\mathcal{Q}_V}. \label{eq:ult1a}%\le C_V\|\varphi-\varphi_\Lambda\|_{\mathcal{Q}_V}, \label{eq:ult1}
\end{align}

Now we estimate $\|\varphi-\varphi_\Lambda\|_{\mathcal{Q}_V}$. To this end, we use Lemma \ref{lm:phi}, which gives
\begin{align*}
%\alpha &= \left(1-\norm{\varphi}_{L^2}^2\right)^{1/2} \\
&\varphi=2R_{\lambda_V}\Pi_V^\perp\,(\bar{\FF}(W)*\abs{u}^2)u,\quad \lambda_V:=E_V-\big\langle u,(\bar{\FF}(W)*|u|^2)u\big\rangle, \\
&\varphi_\Lambda=2R_{\lambda_{V,\Lambda}}\Pi_V^\perp\,(\bar{\FF}(W_\Lambda)*\abs{u_\Lambda}^2)u_\Lambda,\quad \lambda_{V,\Lambda}:=E_{V,\Lambda}-\big\langle u_\Lambda,(\bar{\FF}(W_\Lambda)*|u_\Lambda|^2)u_\Lambda\big\rangle.
\end{align*}
By the triangle inequality,
\begin{align}
\|\varphi-\varphi_\Lambda\|_{\mathcal{Q}_V}&\le T_1+T_2+T_3, \label{eq:ult2}
\end{align}
where
\begin{align*}
&T_1:=2\big\|\big(R_{\lambda_V}\Pi_V^\perp-R_{\lambda_{V,\Lambda}}\Pi_V^\perp\big)(\bar{\FF}(W)*\abs{u}^2)u\big\|_{\mathcal{Q}_V} ,\\
&T_2:=2\big\|R_{\lambda_{V,\Lambda}}\Pi_V^\perp(\bar{\FF}(W-W_\Lambda)*\abs{u}^2)u\big\|_{\mathcal{Q}_V} ,\\
&T_3:=2\big\|R_{\lambda_{V,\Lambda}}\Pi_V^\perp\big((\bar{\FF}(W_\Lambda)*\abs{u}^2)u-(\bar{\FF}(W_\Lambda)*\abs{u_\Lambda}^2)u_\Lambda\big)\big\|_{\mathcal{Q}_V}.
\end{align*}
We first estimate the term $T_3$. As in \eqref{eq:lambda<}, we have that $\lambda_{V,\Lambda}\le\mu_V+\frac12\delta_V$, with the distance to the lower eigenvalue~$\delta_V=\mathrm{dist}(\mu_V,\sigma(H_V)\setminus\{\mu_V\})$. Hence $R_{\lambda_{V,\Lambda}}\Pi_V^\perp$ is in~$\LL(\mathcal{Q}_V^*,\mathcal{Q}_V)$ and
\begin{equation}\label{eq:estim_R2_ultraviolet}
\big\|R_{\lambda_{V,\Lambda}}\Pi_V^\perp\big\|_{\LL(\mathcal{Q}_V^*,\mathcal{Q}_V)}\le2\delta_V^{-1}.
\end{equation}
This yields
\begin{align*}
T_3
&\le 4\delta_V^{-1}\norm{(\bar{\FF}(W)*\abs{u}^2)u-(\bar{\FF}(W)*\abs{u_\Lambda}^2)u_\Lambda}_{\mathcal{Q}_V^*}\\
&\le 4\delta_V^{-1}\norm{(\bar{\FF}(W)*\abs{u}^2)u-(\bar{\FF}(W)*\abs{u_\Lambda}^2)u_\Lambda}_{L^2},
\end{align*}
since $L^2\subset \mathcal{Q}_V^*$. We obtain from the triangle inequality that
\begin{align*}
T_3&\le 4\delta_V^{-1}\norm{\big(\bar{\FF}(W)*[(\bar u -\bar u_\Lambda ) u]\big )u}_{L^2}\\
&\quad+ 4\delta_V^{-1}\norm{\big(\bar{\FF}(W)*[\bar u_\Lambda( u - u_\Lambda )]\big )u}_{L^2}\\
&\quad+ 4\delta_V^{-1}\norm{\big(\bar{\FF}(W)*|u_\Lambda|^2\big )(u-u_\Lambda)}_{L^2} ,
\end{align*}
and hence Lemma \ref{lm:convol-ineq} yields
\begin{align}\label{eq:estim_S2}
T_3&\le C_V\|W\|_{L^1+L^{3,\infty}} \big(\|u\|_{{H}^1}^2+\|u_\Lambda\|_{{H}^1}^2\big) \| u - u_\Lambda\|_{L^2}.
\end{align}
Since in addition we have $\|u_\Lambda\|_{\dot{H}^1}^2\le C_V$ (in the same way as in the proof of Proposition \ref{prop:uniqueness Nelson}) uniformly in $\Lambda$, this gives
\begin{align}
T_3 \le C_V\|W\|_{L^1+L^{3,\infty}} \|u-u_\Lambda\|_{\mathcal{Q}_V}. \label{eq:ult3}
\end{align}

Next we estimate $T_1$. It follows from the resolvent equation that
\begin{equation*}
T_1\le2|\lambda_V-\lambda_{V,\Lambda}|\big\|\big(R_{\lambda_V}\Pi_V^\perp R_{\lambda_{V,\Lambda}}\Pi_V^\perp\big)(\bar{\FF}(W)*\abs{u}^2)u\big\|_{\mathcal{Q}_V}.
\end{equation*}
Using \eqref{eq:estim_R2_ultraviolet}, the fact that, likewise, $\big\|R_{\lambda_{V}}\Pi_V^\perp\big\|_{\LL(\mathcal{Q}_V^*,\mathcal{Q}_V^*)}\le2\delta_V^{-1}$ and then Lemma~\ref{lm:convol-ineq}, we obtain
\begin{equation}\label{eq:estim_S1a}
T_1\le C_V\|W\|_{L^1+L^{3,\infty}} \|u\|_{{H}^1}^2 |\lambda_V-\lambda_{V,\Lambda}| \le C_V\|W\|_{L^1+L^{3,\infty}} |\lambda_V-\lambda_{V,\Lambda}| .
\end{equation}
The expressions of $\lambda$, $\lambda_V$ imply
\begin{align*}
|\lambda_V-\lambda_{V,\Lambda}|
&\le|E_V-E_{V,\Lambda}|+\big|\big\langle u,(\bar{\FF}(W)*|u|^2)u\big\rangle-\big\langle u_\Lambda,(\bar{\FF}(W_\Lambda)*|u_\Lambda|^2)u_\Lambda\big\rangle\big|\notag\\ 
&\le|E_V-E_{V,\Lambda}|+\big|\big\langle u,(\bar{\FF}(W-W_\Lambda)*|u|^2)u\big\rangle\big|\\
&\quad+\big|\big\langle u,(\bar{\FF}(W_\Lambda)*|u|^2)u\big\rangle-\big\langle u_\Lambda,(\bar{\FF}(W_\Lambda)*|u_\Lambda|^2)u_\Lambda\big\rangle\big|.
\end{align*}
The last term can be estimated by the same argument we used to bound $T_3$. This gives
\begin{multline}
|\lambda_V-\lambda_{V,\Lambda}|\\
\le|E_V-E_{V,\Lambda}|+C_V\|W\|_{L^1+L^{3,\infty}} \|u-u_\Lambda\|_{\mathcal{Q}_V}+\big|\big\langle u,(\bar{\FF}(W-W_\Lambda)*|u|^2)u\big\rangle\big|. \label{eq:ult5}
\end{multline}

To estimate the term $T_2$, we write
\begin{align}
T_2&\le C_V \big\|(\bar{\FF}(W-W_\Lambda)*\abs{u}^2)u\big\|_{\mathcal{Q}_V^*} ,\notag\\
&\le C_V \big\|(\bar{\FF}(W-W_\Lambda)*\abs{u}^2)u\big\|_{L^2} ,\label{eq:ult6}
\end{align}
since $L^2\subset \mathcal{Q}_V^*$.

Putting together \eqref{eq:ult1a}, \eqref{eq:ult2}, \eqref{eq:ult3}, \eqref{eq:estim_S1a}, \eqref{eq:ult5} and \eqref{eq:ult6}, we deduce that
\begin{align*}
\big(1-C_V\|W\|_{L^1+L^{3,\infty}}\big)\| u - u_\Lambda\|_{\mathcal{Q}_V}\le C_V\big( &|E_V-E_{V,\Lambda}|+\big|\big\langle u,(\bar{\FF}(W-W_\Lambda)*|u|^2)u\big\rangle\big| \\
&+\big\|(\bar{\FF}(W-W_\Lambda)*\abs{u}^2)u\big\|_{L^2}\big).
\end{align*}
For $\|W\|_{L^1+L^{3,\infty}}\le\varepsilon_V$ with $\varepsilon_V$ small enough, Lemma \ref{lm:Leb} together with Proposition \ref{prop:conv-GS-en} then imply that $\|u-u_\Lambda\|_{\mathcal{Q}_V}\to0$ as $\Lambda\to\infty$.
\end{proof}
%
%

%\newpage

\appendix

\section{Operators in Fock space, self-adjointness}\label{app:Fock}

\subsection{Operators in Fock space}

We recall in this section a few well-known properties of basic operators in Fock space. We do not specify their domains. For more details the reader may consult e.g.~\cite{Berezin66,BratteliRobinson297,ReedSimonII}. Recall that the symmetric Fock space $\F(\mathfrak{h})$ over the one-particle space~$\mathfrak{h}= L^2(\mathbb{R}^3)$ has been defined in \eqref{eq:FockSpace}. For $h$ in~$\mathfrak{h},$ the creation and annihilation operators~$a^*(h)$ and~$a(h)$  are defined as follows:
\begin{align*}
a^*(h)_{|\bigvee^n\mathfrak{h}}&=\sqrt{(n+1)}\;|h\rangle\bigvee\I_{\bigvee^n\mathfrak{h}}, \quad n \ge 0, \\
a(h)_{|\bigvee^n\mathfrak{h}}&=\sqrt{n}\;\langle h|\otimes\I_{\bigvee^{n-1}\mathfrak{h}} , \quad n > 0 , \qquad a(h)_{|\mathbb{C}}=0.
\end{align*}
Formally, we also have
\begin{equation*}
a(h)=\int_{\mathbb{R}^3} \overline{h(k)}a(k)\diff k,\quad a^*(h)=\int_{\mathbb{R}^3} h(k)a^*(k)\diff k,
\end{equation*} 
where $a(k)$ and $a^*(k)$ are operator-valued distributions which satisfy the well-known canonical commutations relations
\begin{align*}
[a(k),a(k')]&=[a^*(k),a^*(k')] = 0,\quad 
[a(k),a^*(k')]=\delta(k-k').
\end{align*}
The field operator $\Phi(h)$ is defined by
\begin{equation*}
\Phi(h)=(a(h)+a^*(h))/\sqrt{2}.
\end{equation*}
Let $\omega$ be a self-adjoint operator on $\mathfrak{h}.$ The second quantization of $\omega$ is defined by 
\begin{equation*}
\mathrm{d}\Gamma(\omega)_{|\bigvee^n\mathfrak{h}}=\sum_{k=1}^n\I_{\bigvee^{k-1}\mathfrak{h}}\otimes \omega\otimes \I_{\bigvee^{n-k}\mathfrak{h}}.
\end{equation*} 
Note that this operator can be expressed in terms of creation and annihilation operators :
\begin{equation*}
\mathrm{d}\Gamma(\omega)=\int_{\mathbb{R}^3} \omega(k) a^*(k)a(k)\diff k.
\end{equation*}
The coherent state of parameter $f$ in~$\mathfrak{h}$ is the vector in Fock space defined as
\begin{equation*}
\Psi_f:=e^{i\Phi\left(\frac{\sqrt{2}}{i}f\right)}\Omega=e^{-\frac{\norm{f}^2_{\mathfrak{h}}}{2}}\sum_{n=0}^\infty\frac{f^{\otimes n}}{\sqrt{n!}},
\end{equation*}
where $\Omega$ stands for the Fock vacuum. Coherent states are eigenvectors of the annihilation operator in the sense that, for all $f,h$ in~$\mathfrak{h},$  we have 
\begin{equation*}
a(h)\Psi_f=\ps{h}{f}_{\mathfrak{h}}\Psi_f.
\end{equation*}
This identity implies the following relations:
\begin{align*}
\ps{\Psi_f}{\Phi(h)\Psi_f}_{\F(\mathfrak{h})}&=2\Real\ps{h}{f}_\mathfrak{h}, \qquad \ps{\Psi_f}{\mathrm{d}\Gamma(\omega)\Psi_f}_{\F(\mathfrak{h})}=\ps{f}{\omega\, f}_\mathfrak{h}.
\end{align*}
These equalities were used to compute the expressions \eqref{eq:cE}--\eqref{eq:cE2} of the Pauli-Fierz energy of product stated of the form $u\otimes\Psi_f$.

\subsection{Self-adjointness of the Pauli-Fierz Hamiltonian} In the next proposition, we recall the self-adjointness property of the Pauli-Fierz Hamiltonian $\mathbb{H}$ defined in \eqref{eq:general_PF}. Recall also that the free Hamiltonian $\mathbb{H}_{\mathrm{free}}$ has been defined in \eqref{eq:Hfree}.
\begin{proposition}
Suppose $V$ satisfies Hypothesis \ref{condV} and that $W=g^2\omega^{-1}v^2$ is in~$L^1(\mathbb{R}^3)$. Then $\mathbb{H}$ is a self-adjoint, semi-bounded operator with form domain $\mathcal{Q}(\mathbb{H})=\mathcal{Q}(\mathbb{H}_{\mathrm{free}})$ for all~$g$ in~$\mathbb{R}$.
\end{proposition}
\begin{proof}
Using the well-known $N_\tau$-estimates for the creation and annihilation operators, we have
\begin{align*}
\norm{a(h_x)\psi}_{\F(\mathfrak{h})}&\leq\norm{\omega^{-1/2}h_x}_\mathfrak{h}\norm{\mathrm{d}\Gamma(\omega)^{1/2}\psi}_{\F(\mathfrak{h})},\\
\norm{a^*(h_x)\psi}_{\F(\mathfrak{h})}&\leq\norm{\omega^{-1/2}h_x}_\mathfrak{h}\norm{\mathrm{d}\Gamma(\omega)^{1/2}\psi}_{\F(\mathfrak{h})}+\norm{h_x}_\mathfrak{h}\norm{\psi}_{\F(\mathfrak{h})}.
\end{align*}
Note that the quantity $\norm{\omega^{-1/2}h_x}_\mathfrak{h}$ is well-defined since $W$ is in~$L^1$. Now, by the Cauchy-Schwarz inequality, we have
\begin{align*}
\norm{\mathrm{d}\Gamma(\omega)^{1/2}\psi}_{\F(\mathfrak{h})}^2&=\ps{\psi}{\mathrm{d}\Gamma(\omega)\psi}_{\F(\mathfrak{h})}\\
&\leq \norm{\psi}_{\F(\mathfrak{h})}\norm{\mathrm{d}\Gamma(\omega)\psi}_{\F(\mathfrak{h})} \leq \frac{1}{2\varepsilon^2}\norm{\psi}^2_{\F(\mathfrak{h})}+\frac{\varepsilon^2}{2}\norm{\mathrm{d}\Gamma(\omega)\psi}_{\F(\mathfrak{h})}^2.
\end{align*}
Taking $\varepsilon>0$ small enough allows us to conclude that there exists $a<1$ and $b$ in~$\R$ such that, for all $\psi$ in~$\F(\mathfrak{h}),$ the following inequality holds 
\begin{equation*}
\norm{g\Phi(h_x)\psi}_{\F(\mathfrak{h})}\leq a\norm{\mathbb{H}_{\mathrm{free}}\psi}_{\F(\mathfrak{h})}+b\norm{\psi}_{\F(\mathfrak{h})}
\end{equation*}
Therefore $g\Phi(h_x)$ is relatively bounded (and hence also relatively form bounded) with respect to $\mathbb{H}_{\mathrm{free}}$ with relative bound less than $1$. Applying the KLMN theorem  (see e.g.~\cite[Theorem X.17]{ReedSimonII}) then yields the result.
\end{proof}

As mentioned in the introduction, the condition $W\in L^1$ is not satisfied by the polaron model. It is nevertheless proven in \cite{GriesemerWunsch16}, by other means, that the polaron Hamiltonian $\mathbb{H}$ also identifies with a semi-bounded self-adjoint operator with form domain $\mathcal{Q}(\mathbb{H})=\mathcal{Q}(\mathbb{H}_{\mathrm{free}})$.

\section{Existence of a minimizer for the Hartree equation}\label{app:Hartree}

In this section, we prove the existence of a minimizer for the Hartree energy functional as stated in Proposition \ref{prop:min_Hartree}. We write $w=\bar{\FF}(W)=g^2\bar{\FF}(\omega^{-1}v^2)$ (where, recall, $\omega$ is the field dispersion relation, $v$ is the coupling function and $g$ is the coupling parameter) and display the dependence of the Hartree energy functional \eqref{eq:Hartree_en} on the external potential $V$. In other words, we study in this section the energy functional
\begin{align*}
J_V(u) :=  \langle u,H_{V}u\rangle-\int_{\mathbb{R}^3} \big(w*|u|^2\big)(x) |u(x)|^2 \diff x, \quad u \in \mathcal{U},
\end{align*}
where $V,w:\mathbb{R}^3\to\mathbb{R}$ are real potentials, $H_{V} = -\Delta + V$, and
\begin{equation*}
\mathcal{U} = \big\{ u \in \mathcal{Q}_{V}\mid \|u\|_{L^2}=1\big\},
\end{equation*}
with $\mathcal{Q}_{V}\subset L^2(\mathbb{R}^3)$ the form domain of $H_V$.
%We establish the existence of a minimizer for $J_V$ under suitable conditions on $V$ and $w$. Recall that in Section \ref{sec:linear}, the convolution potential is given by $w=\bar{\FF}(W)=g^2\bar{\FF}(\omega^{-1}v^2)$, with $\omega$ the field dispersion relation, $v$ the coupling function and $g$ the coupling parameter.
%
%%Our strategy rests on usual arguments from the calculus of variations \cite{Li84_01,Li84_02}. Existence of minimizers for the Hartree energy has been studied by many authors in different contexts (see e.g.~{\color{blue}REF}). We are not aware, however, of a result giving the existence of a minimizer under our conditions on $V$ and $W$ (see Hypotheses \ref{condV} and \ref{condv0a} below). Therefore we provide a complete proof.
%
%The `external' potential $V$ is supposed to satisfy Hypothesis~\ref{condV}, while the `pair' potential $w$ is supposed to satisfy the following hypothesis (which corresponds to Hypothesis \ref{condv0:2a} in the case where $w=\bar{\FF}(W)$).
%\begin{assumption2}{3'}[Conditions on $w$]\label{condv0}The `pair' potential $w$ is a tempered distribution such that $W=\FF(w)$ decomposes as $W=W_1+W_2$ with
%\begin{enumerate}[label=(\roman*)]
%\item $W_1\in L^1(\mathbb{R}^3;\mathbb{R})$,
%\item $W_2\in L^{3,\infty}(\mathbb{R}^3;\mathbb{R})$.
%\end{enumerate}
%\end{assumption2}

We begin with a lemma showing that $J_V$ is well-defined and semi-bounded from below under our assumptions.%, provided that $\|W_2\|_{L^{3,\infty}}$ is small enough.
\begin{lemma}\label{lm:JV}
Assume that $V$ satisfies Hypothesis \ref{condV} and that $W$ satisfies Hypothesis \ref{condv0:2a}. Then~$J_V(u)$ is well-defined for all $ u $ in~$ \mathcal{U}$. Moreover, if the decomposition $W=W_1+W_2$ in Hypothesis \ref{condv0:2a} can be chosen such that
\begin{equation*}
 \|W_2\|_{L^{3,\infty}} < C(1-a)
\end{equation*}
for some universal constant $C$, where $a$ is as in Hypothesis \ref{condV}, then
\begin{align}\label{eq:exist_EV}
E_V := \inf_{u\in \mathcal{U}} J_V(u) > - \infty.
\end{align}
\end{lemma}
\begin{proof}
Combining \eqref{eq:convol-ineq1}--\eqref{eq:convol-ineq2} from Lemma \ref{lm:convol-ineq} and \eqref{eq:estimH1} from Lemma \ref{lm:relative-bound}, we deduce that, for $u$ in~$\mathcal{U}$,
\begin{equation}\label{eq:JV<}
J_V(u)\le \Big(1+\frac{C\|W_2\|_{L^{3,\infty}}}{1-a}\Big)\langle u,H_Vu\rangle+\Big(\|W_1\|_{L^1}+\frac{ b C\|W_2\|_{L^{3,\infty}}}{1-a}\Big),
\end{equation}
for some universal constant $C$, and
\begin{equation}\label{eq:JV>}
J_V(u) \ge \Big(1-\frac{C\|W_2\|_{L^{3,\infty}}}{1-a}\Big)\langle u,H_Vu\rangle-\Big(\|W_1\|_{L^1}+\frac{ b C\|W_2\|_{L^{3,\infty}}}{1-a}\Big).
\end{equation}
Therefore, assuming $C\|W_2\|_{L^{3,\infty}}<1-a$, we deduce that \eqref{eq:exist_EV} holds.
\end{proof}

Next we prove Proposition \ref{prop:min_Hartree}. Recall that $\mu_V = \inf \sigma( H_V )$. Below, if $\tilde V$ is another potential, we denote by $J_{\tilde V}$, $E_{\tilde V}$, $\mu_{\tilde V}$ the quantities obtained from $J_V$, $E_V$, $\mu_V$ by replacing $V$ by $\tilde V$. 

\begin{proof}[Proof of Proposition \ref{prop:min_Hartree}]
Let $(u_j)_{j\in\mathbb{N}}\subset \mathcal{U}$ be a minimizing sequence for $J_V$, i.e.
\begin{equation*}
J_V(u_j) \underset{j\to+\infty}{\longrightarrow} E_V.
\end{equation*}
The strategy consists in showing that $(u_j)_{j\in\mathbb{N}}$ converges strongly in $L^2(\mathbb{R}^3)$, along some subsequence, to a state $u_\infty$ in~$\mathcal{Q}_V$ which is then a minimizer for $H_V$.
%To shorten expressions below, the notation $\mathcal{O}(\|W_2\|_{L^{3,\infty}})$ stands for a term proportional to $\|W_2\|_{L^{3,\infty}}$.
We divide the proof into several steps.

\medskip

\noindent \textbf{Step 1}: We prove that 
\begin{equation}\label{eq:step1}
E_V<E_{V_1}.
\end{equation}
Let $\varepsilon>0$. Let $u_\varepsilon$ in~$\mathcal{U}$ be such that $\langle u_\varepsilon,H_Vu_\varepsilon\rangle\le \mu_V+\varepsilon$. Applying \eqref{eq:JV<}, we obtain that
\begin{align*}
E_V&\le J_V(u_\varepsilon) \le \Big(1+\frac{C\|W_2\|_{L^{3,\infty}}}{1-a}\Big)\big(\mu_V+\varepsilon\big)+\|W_1\|_{L^1}+\frac{ b C\|W_2\|_{L^{3,\infty}}}{1-a}.
\end{align*}
Letting $\varepsilon\to0$, this yields
\begin{equation}\label{eq:r1}
E_V\le \Big(1+\frac{C\|W_2\|_{L^{3,\infty}}}{1-a}\Big)\mu_V+\|W_1\|_{L^1}+\frac{b C\|W_2\|_{L^{3,\infty}}}{1-a}.
\end{equation}
On the other hand, for all $u$ in~$\mathcal{U}$, we can write, using \eqref{eq:convol-ineq1}, \eqref{eq:convol-ineq2} and the fact that $V_1\ge0$,
\begin{align*}
J_{V_1}(u) &\ge \langle u,H_{V_1}u\rangle - \|W_1\|_{L^1} - C\|W_2\|_{L^{3,\infty}}\|u\|_{\dot{H}_1}^2 \\
&\ge \big(1-C\|W_2\|_{L^{3,\infty}}\big) \mu_{V_1}-\|W_1\|_{L^1},
\end{align*}
and hence
\begin{equation}\label{eq:r2}
E_{V_1}\ge \big(1-C\|W_2\|_{L^{3,\infty}}\big) \mu_{V_1}-\|W_1\|_{L^1}.
\end{equation}
Combining \eqref{eq:r1} and \eqref{eq:r2} gives
\begin{align}
E_{V_1}-E_V &\ge \big(1-C\|W_2\|_{L^{3,\infty}}\big) \mu_{V_1}-\Big(1+\frac{C\|W_2\|_{L^{3,\infty}}}{1-a}\Big)\mu_V - 2\|W_1\|_{L^1} - \frac{ b C\|W_2\|_{L^{3,\infty}}}{1-a} \notag \\
&= \big(1-C\|W_2\|_{L^{3,\infty}}\big) (\mu_{V_1}- \mu_V) - 2\|W_1\|_{L^1}-C_V\|W_2\|_{L^{3,\infty}} , \label{eq:e1}
\end{align}
where we have set
\begin{align*}
C_V := C \frac{(2-a)\mu_V+b}{1-a} .
\end{align*}
The right-hand-side of \eqref{eq:e1} is strictly positive since we assumed $C\|W_2\|_{L^{3,\infty}}\le\frac12(1-a)\le\frac12$ and provided that
\begin{equation*}
2\|W_1\|_{L^1}+C_V\|W_2\|_{L^{3,\infty}}<\frac14(\mu_{V_1}-\mu_V).
\end{equation*}

\medskip

\noindent \textbf{Step 2}:  We prove that for all $u$ in~$\mathcal{U}$,
\begin{equation}
J_V(u)\ge E_V + \big(E_{V_1}-E_V-4\|W_1\|_{L^1}-C\|W_2\|_{L^{3,\infty}}\|u\|_{H^1}^2 \big)\|\tilde{\eta}_Ru\|^2_{L^2}+o(R^0), \label{eq:a0}
\end{equation}
as $R\to\infty$, for some universal constant $C>0$.

Recall that the localizations functions $\eta_R$, $\tilde{\eta}_R$ have been defined in \eqref{eq:defetaR}. Writing the decomposition $u=\eta_R^2u+\tilde{\eta}_R^2u$ and commuting $\eta_R$, $\tilde{\eta}_R$ through $-\Delta$, we have by the IMS localization formula (see e.g.~\cite{CFKS87})
\begin{align}\label{eq:h1}
\langle u,H_{V}u\rangle = \langle \eta_Ru,H_{V}\eta_Ru\rangle+\langle \tilde{\eta}_Ru,H_{V}\tilde{\eta}_Ru\rangle - \frac12\big\langle u,\big(\big|\nabla\eta_R\big|^2+\big|\nabla\tilde{\eta}_R\big|^2\big)u\big\rangle.
\end{align}
The definitions of $\eta_R$, $\tilde{\eta}_R$ yield $\||\nabla\eta_R|^2\|_{L^\infty}=\mathcal{O}(R^{-2})$, $\||\nabla\tilde{\eta}_R|^2\|_{L^\infty}=\mathcal{O}(R^{-2})$. Moreover, since $V_2(x)\to0$ as $|x|\to\infty$, 
\begin{equation*}
\langle \tilde{\eta}_Ru,H_{V}\tilde{\eta}_Ru\rangle=\langle \tilde{\eta}_Ru,H_{V_1}\tilde{\eta}_Ru\rangle+o(R^0)\|\tilde{\eta}_Ru\|^2.
\end{equation*}
Inserting this into \eqref{eq:h1}, we get
\begin{align}\label{eq:h2}
\langle u,H_{V}u\rangle = \langle \eta_Ru,H_{V}\eta_Ru\rangle+\langle \tilde{\eta}_Ru,H_{V_1}\tilde{\eta}_Ru\rangle +o(R^0).
\end{align}

Next we consider the convolution term in $J_V(u)$. We can write
\begin{multline}
\int_{\mathbb{R}^3} \big(w*|u|^2\big)(x) |u(x)|^2 \diff x\\
=\int_{\mathbb{R}^3} \big(w*|\eta_Ru|^2\big)(x) |(\eta_Ru)(x)|^2 \diff x+\int_{\mathbb{R}^3} \big(w*|\tilde{\eta}_Ru|^2\big)(x) |(\tilde{\eta}_Ru)(x)|^2 \diff x \\
\quad+\int_{\mathbb{R}^3} \big(w*|\eta_Ru|^2\big)(x) |(\tilde{\eta}_Ru)(x)|^2 \diff x+\int_{\mathbb{R}^3} \big(w*|\tilde{\eta}_Ru|^2\big)(x) |(\eta_Ru)(x)|^2 \diff x. \label{eq:k1}
\end{multline}
We estimate the last two terms in the right-hand-side of the previous equation. Since $W_1$ is in~$L^1(\mathbb{R}^3)$, using \eqref{eq:convol-ineq1} yields
\begin{multline}
\Big |\int_{\mathbb{R}^3} \big(\bar\FF(W_1)*|\eta_Ru|^2\big)(x) |(\tilde{\eta}_Ru)(x)|^2 \diff x+\int_{\mathbb{R}^3} \big(\bar\FF(W_1)*|\tilde{\eta}_Ru|^2\big)(x) |(\eta_Ru)(x)|^2 \diff x\Big |\\
\le 2\|W_1\|_{L^1}\|\eta_Ru\|_{L^2}^2\|\tilde{\eta}_Ru\|_{L^2}^2\le2\|W_1\|_{L^1}\|\tilde{\eta}_Ru\|_{L^2}^2.\label{eq:k2}
\end{multline}
Note that in the last inequality we used that $\|\eta_Ru\|_{L^2}\le\|u\|_{L^2}=1$. Next, since $W_2$ is in~$L^{3,\infty}(\mathbb{R}^3)$, using \eqref{eq:convol-ineq2} yields
\begin{multline}
\Big |\int_{\mathbb{R}^3} \big(\bar\FF(W_2)*|\eta_Ru|^2\big)(x) |(\tilde{\eta}_Ru)(x)|^2 \diff x+\int_{\mathbb{R}^3} \big(\bar\FF(W_2)*|\tilde{\eta}_Ru|^2\big)(x) |(\eta_Ru)(x)|^2 \diff x\Big |\\
\le C \|W_2\|_{L^{3,\infty}}\|\eta_Ru\|_{H_1}^2\|\tilde{\eta}_Ru\|_{L^2}^2\le C\|W_2\|_{L^{3,\infty}}\|u\|_{H^1}^2\|\tilde{\eta}_Ru\|_{L^2}^2+o(R^0),\label{eq:k3}
\end{multline}
where in the last inequality we used in addition that $\|\eta_Ru\|_{H_1}\le\|u\|_{H^1}+o(R^0)$. Putting together \eqref{eq:h2}, \eqref{eq:k1}, \eqref{eq:k2} and \eqref{eq:k3}, we arrive at
\begin{align}
J_V(u)\ge J_V(\eta_Ru)+J_{V_1}(\tilde{\eta}_Ru) - \big(2\|W_1\|_{L^1}+C\|W_2\|_{L^{3,\infty}}\|u\|_{H^1}^2\big)\|\tilde{\eta}_Ru\|_{L^2}^2+o(R^0). \label{eq:a1}
\end{align}
Now, suppose that $\|\eta_Ru\|_{L^2}\neq0$. Then we can write
\begin{align*}
J_V(\eta_Ru)&=\Big\langle \frac{\eta_Ru}{\|\eta_Ru\|_{L^2}},H_V\frac{\eta_Ru}{\|\eta_Ru\|_{L^2}}\Big\rangle \, \|\eta_Ru\|_{L^2}^2\\
&\quad-\int_{\mathbb{R}^3} \Big(w*\Big|\frac{\eta_Ru}{\|\eta_Ru\|_{L^2}}\Big|^2\Big)(x) \, \Big|\frac{\eta_Ru}{\|\eta_Ru\|_{L^2}}(x)\Big|^2 \diff x\,\|\eta_Ru\|_{L^2}^4 \notag \\
&=J_V\Big(\frac{\eta_Ru}{\|\eta_Ru\|_{L^2}}\Big) \|\eta_Ru\|_{L^2}^2\\
&\quad+\int_{\mathbb{R}^3} \Big(w*\Big|\frac{\eta_Ru}{\|\eta_Ru\|_{L^2}}\Big|^2\Big)(x) \, \Big|\frac{\eta_Ru}{\|\eta_Ru\|_{L^2}}(x)\Big|^2 \diff x \,\|\eta_Ru\|_{L^2}^2 \, \|\tilde{\eta}_Ru\|_{L^2}^2 .
\end{align*}
By definition of $E_V$, we have that $J_V(\eta_Ru/\|\eta_Ru\|_{L^2})\ge E_V$. Estimating the integrated term as above, using \eqref{eq:convol-ineq1}, \eqref{eq:convol-ineq2} and $\|\eta_Ru\|_{H_1}\le\|u\|_{H^1}+o(R^0)$, we then deduce that
\begin{align}
&J_V(\eta_Ru) \ge E_V\|\eta_Ru\|_{L^2}^2-\big(\|W_1\|_{L^1}+C\|W_2\|_{L^{3,\infty}}\|u\|_{H^1}^2\big)\|\tilde{\eta}_Ru\|_{L^2}^2+o(R^0). \label{eq:a2}
\end{align}
Similar arguments show that, if $\|\tilde{\eta}_Ru\|\neq0$, then
\begin{align}
J_{V_1} &( \bar \eta_R u ) \notag\\
&\ge E_{V_1} \| \bar \eta_R u \|_{L^2}^2 - \big(\|W_1\|_{L^1}+C\|W_2\|_{L^{3,\infty}}\|u\|_{H^1}^2\big) \|\tilde{\eta}_Ru\|_{L^2}^2 +o(R^0) \notag \\
&= \big(E_{V_1}-E_V -  \|W_1\|_{L^1}-C\|W_2\|_{L^{3,\infty}}\|u\|_{H^1}^2\big) \|\tilde{\eta}_Ru\|_{L^2}^2+E_V\|\tilde{\eta}_Ru\|_{L^2}^2+o(R^0). \label{eq:a3}
\end{align}
We claim that \eqref{eq:a1}, \eqref{eq:a2} and \eqref{eq:a3} imply \eqref{eq:a0}. Indeed, if $\|\tilde{\eta}_Ru\|_{L^2}=0$, then $\|\eta_Ru\|_{L^2}=1$ and \eqref{eq:a0} follows from \eqref{eq:a1} and \eqref{eq:a2}. If $\|\eta_Ru\|_{L^2}=0$, then $\|\tilde{\eta}_Ru\|_{L^2}=1$ and \eqref{eq:a0} follows from \eqref{eq:a1} and \eqref{eq:a3}. Finally if both $\|\eta_Ru\|_{L^2}\neq0$ and $\|\tilde{\eta}_Ru\|_{L^2}\neq0$, then combining \eqref{eq:a1}, \eqref{eq:a2} and \eqref{eq:a3} gives \eqref{eq:a0}, since $\|\eta_Ru\|_{L^2}^2+\|\tilde{\eta}_Ru\|_{L^2}^2=1$.

\medskip

\noindent \textbf{Step 3}: We prove that $(u_j)_{j\in\mathbb{N}}$ is bounded in $\mathcal{Q}_V$ (equipped with the norm defined in \eqref{eq:normQ}), uniformly in $W_2$ such that $C\|W_2\|_{L^{3,\infty}}\le \frac12(1-a)$, for some universal constant $C$.

It suffices to show that $(\langle u_j,-\Delta u_j\rangle+\langle u_j,V_+u_j\rangle)_{j\in\mathbb{N}}=(\langle u_j,H_{V_+}u_j\rangle)_{j\in\mathbb{N}}$ is bounded. We proceed as in the proof of Lemma \ref{lm:JV}. We write
\begin{align}\label{eq:b1}
\langle u_j,H_{V_+} u_j\rangle = J_V(u_j) +\langle u_j,V_-u_j\rangle+\int_{\mathbb{R}^3} \big(w*|u_j|^2\big)(x) |u_j(x)|^2 \diff x.
\end{align}
For the second term in the right-hand-side of the previous equation, we use Hypothesis \ref{condV}, which implies that
\begin{align}\label{eq:b-1}
 \langle u_j,V_-u_j\rangle \le a\langle u_j,H_{V_+} u_j\rangle+b\|u_j\|_{L^2}^2.
\end{align}
The third term of the right-hand-side of \eqref{eq:b1} can be estimated using \eqref{eq:convol-ineq1} and \eqref{eq:convol-ineq2}, namely
\begin{align}
\Big | \int_{\mathbb{R}^3} \big(w*|u_j|^2\big)(x) |u_j(x)|^2 \diff x \Big | &\le \|W_1\|_{L^1}\|u_j\|_{L^2}^4+C\|W_2\|_{L^{3,\infty}}\|u_j\|_{\dot{H}^1}^2\|u_j\|_{L^2}^2\notag\\
&\le \|W_1\|_{L^1}+C\|W_2\|_{L^{3,\infty}}\langle u_j,H_{V_+}u_j\rangle,\label{eq:b-2}
\end{align}
since $\|u_j\|_{L^2}=1$. Inserting \eqref{eq:b-1}--\eqref{eq:b-2} into \eqref{eq:b1}, assuming that $C\|W_2\|_{L^{3,\infty}}<1-a$, we obtain
\begin{align}\label{eq:b2}
\langle u_j,H_{V_+}u_j\rangle\le \frac{1}{1-a-C\|W_2\|_{L^{3,\infty}}} \Big ( J_V(u_j) + b + \|W_1\|_{L^1}\Big) .
\end{align}
Since $(J_V(u_j))_{j\in\mathbb{N}}$ converges, it is bounded. This proves that $(u_j)_{j\in\mathbb{N}}$ is bounded in $\mathcal{Q}_V$. In turn, since we can assume without loss of generality that $J_V(u_j)\le E_V+1$ for all $j$, one easily concludes form the previous equation together with \eqref{eq:r1} that $(u_j)_{j\in\mathbb{N}}$ is bounded in~$\mathcal{Q}_V$ uniformly in $W_2$ such that $C\|W_2\|_{L^{3,\infty}}\le \frac12(1-a)$.

\medskip

\noindent \textbf{Step 4}:  By Step 3, we know that $(u_j)_{j\in\mathbb{N}}$ is bounded in $\mathcal{Q}_V$. Hence there exists a subsequence, still denoted $(u_j)_{j\in\mathbb{N}}$, which converges weakly in $\mathcal{Q}_V$. Let $u_\infty\in \mathcal{Q}_V$ be its limit. In particular $u_j\to u_\infty$ weakly in $H^1(\mathbb{R}^3)$. We prove that $(u_j)_{j\in\mathbb{N}}$ converges strongly (along some subsequence) to $u_\infty$ in $L^2(\mathbb{R}^3)$.

As in Step 2, we write, for $R>0$,
\begin{equation}
\|u_j-u_\infty\|_{L^2}^2=\|\eta_R(u_j-u_\infty)\|_{L^2}^2+\|\tilde{\eta}_R(u_j-u_\infty)\|_{L^2}^2. \label{eq:c1}
\end{equation}
Consider first the term $\|\tilde{\eta}_R(u_j-u_\infty)\|_{L^2}^2$. Let $\varepsilon>0$. It follows from Step 2 that there exists~$R_0>0$ such that, for $R\ge R_0$,
\begin{equation}\label{eq:t1}
\|\tilde{\eta}_Ru_j\|_{L^2}^2\le \frac{J_V(u_j)- E_V}{E_{V_1}-E_V-4\|W_1\|_{L^1}-C\|W_2\|_{L^{3,\infty}}\|u_j\|_{H^1}^2}+\varepsilon.
\end{equation}
Here it should be noted that the term $E_{V_1}-E_V-4\|W_1\|_{L^1}-C\|W_2\|_{L^{3,\infty}}\|u_j\|_{H^1}^2$ is strictly positive. Indeed, we know from Step 3 that $(\|u_j\|_{H^1})_{j\in\mathbb{N}}$ is bounded uniformly in $W_2$ such that $C\|W_2\|_{L^{3,\infty}}\le \frac12(1-a)$. Together with \eqref{eq:e1}, this shows that
\begin{align*}
E_{V_1}-E_V-4\|W_1\|_{L^1}-C\|W_2\|_{L^{3,\infty}}\|u_j\|_{H^1}^2\ge \frac12(\mu_{V_1}-\mu_V)-6\|W_1\|_{L^1}- (C+C_V)\|W_2\|_{L^{3,\infty}},
\end{align*}
By the assumption \eqref{eq:condv:0}, the right-hand-side of the previous equation is strictly positive. Returning now to \eqref{eq:t1}, using in addition that $J_V(u_j)\to E_V$, we deduce that there exists~$j_0$ in~$\mathbb{N}$ such that, for all $j\ge j_0$ (and $R\ge R_0$),
\begin{equation}\label{eq:c2}
\|\tilde{\eta}_Ru_j\|_{L^2}^2\le 2\varepsilon.
\end{equation}
Using the lower semi-continuity of $\|\cdot\|_{L^2}$, we also obtain that, for $R\ge R_0$, 
\begin{equation}\label{eq:c3}
\|\tilde{\eta}_Ru_\infty\|_{L^2}^2\le\liminf_{j\to\infty}\|\tilde{\eta}_Ru_j\|_{L^2}^2\le 2\varepsilon.
\end{equation}

Now fix $R_0>0$ such that \eqref{eq:c2}--\eqref{eq:c3} hold and consider the term $\|\eta_{R_0}(u_j-u_\infty)\|_{L^2}^2$ from \eqref{eq:c1}.  Clearly, since $(u_j)_{j\in\mathbb{N}}$ converges weakly to $u_\infty$ in $H^1(\mathbb{R}^3)$, it follows that $(\eta_{R_0}u_j)_{j\in\mathbb{N}}$ converges weakly to $\eta_{R_0}u_\infty$ in $H^1(B_{2R_0})$, where $B_{2R_0}=\{x\in\mathbb{R}^3\, | \, |x|\le 2R_0\}$. The Rellich-Kondrachov Theorem then gives the existence of a subsequence, still denoted by $(\eta_{R_0}u_j)_{j\in\mathbb{N}}$, which converges strongly to $\eta_{R_0}u_\infty$ in $L^2(B_{2R_0})$. We can then conclude that there exists an integer~$j_1\ge j_0$, such that, for all $j\ge j_1$,
\begin{equation}
\|u_j-u_\infty\|_{L^2}^2=\|\eta_{R_0}(u_j-u_\infty)\|_{L^2}^2+\|\tilde{\eta}_{R_0}(u_j-u_\infty)\|_{L^2}^2 \le \varepsilon + 8\varepsilon.
\end{equation}
Hence $(u_j)_{j\in\mathbb{N}}$ converges strongly to $u_\infty$ in $L^2(\mathbb{R}^3)$.

\medskip

\noindent \textbf{Step 5}:  We prove that $u_\infty$ is a minimizer for $J_V$.

Obviously, since $u_j\to u_\infty$ strongly in $L^2(\mathbb{R}^3)$, we have that $\|u_\infty\|_{L^2}=1$. Moreover, we clearly have
\begin{equation*}
E_V\le J_V(u_\infty).
\end{equation*}
Hence it remains to show that $J_V(u_\infty)\le E_V$. 

Recall that
\begin{equation}\label{eq:d0}
J_V(u_\infty)=\langle u_\infty,(-\Delta + V) u_\infty\rangle - \int_{\mathbb{R}^3} \big(w*|u_\infty|^2\big)(x) |u_\infty(x)|^2 \diff x.
\end{equation}
By Step~3, $(u_j)$ is bounded in $H^1(\mathbb R^3)$ and by Step~4, $(u_j)$ converges strongly to $u_\infty$ in $L^2(\mathbb R^3)$. Hence Lemma~\ref{lem:LowerSemiContDeltaV} yields
\[
\langle u_{\infty},(-\Delta+V)u_{\infty}\rangle\leq\liminf_{j\to\infty}\langle u_{j},(-\Delta+V)u_{j}\rangle\,.
\]

 It remains to consider the quartic term in \eqref{eq:d0}. As in \eqref{eq:convol-ineq1}--\eqref{eq:convol-ineq2}, we have
\begin{align*}
\Big | \int_{\mathbb{R}^3}& \big(w*|u_j|^2\big)(x) |u_j(x)|^2 \diff x - \int_{\mathbb{R}^3} \big(w*|u_\infty|^2\big)(x) |u_\infty(x)|^2 \diff x\Big |\\
&=\Big | \int_{\mathbb{R}^3} \big(w*(|u_j|^2-|u_\infty|^2)\big)(x) |u_j(x)|^2 \diff x + \int_{\mathbb{R}^3} \big(w*|u_\infty|^2\big)(x) (|u_j(x)|^2-|u_\infty(x)|^2 \diff x\Big |\\
&\lesssim \big(\|W_1\|_{L^1}+\|W_2\|_{L^{3,\infty}}\big)\big(\|u_j\|_{H^1}^2+\|u_\infty\|_{H^1}^2\big) \big\||u_j|^2-|u_\infty|^2\big\|_{L^1}\\
&\lesssim \big(\|W_1\|_{L^1}+\|W_2\|_{L^{3,\infty}}\big)\big(\|u_j\|_{H^1}^2+\|u_\infty\|_{H^1}^2\big)\|u_j+u_\infty\|_{L^2}\|u_j-u_\infty\|_{L^2}\\
&\lesssim \big(\|W_1\|_{L^1}+\|W_2\|_{L^{3,\infty}}\big)\|u_j-u_\infty\|_{L^2},
\end{align*}
where we used in the last inequality that $\|u_j\|_{L^2}=1$ and that $(u_j)_{j\in\mathbb{N}}$ is bounded in $H^1(\mathbb{R}^3)$. Since $u_j\to u_\infty$ strongly in $L^2(\mathbb{R}^3)$, this yields
\begin{equation}\label{eq:d6}
\int_{\mathbb{R}^3} \big(w*|u_\infty|^2\big)(x) \, |u_\infty(x)|^2 \diff x=\lim_{j\to\infty} \int_{\mathbb{R}^3} \big(w*|u_j|^2\big)(x) \, |u_j(x)|^2 \diff x .
\end{equation}

Inserting \eqref{eq:d1}, \eqref{eq:d2}, \eqref{eq:d3} and \eqref{eq:d6} into \eqref{eq:d0}, we finally obtain that
\begin{equation*}
J_V(u_\infty)\le\liminf_{j\to\infty}J_V(u_j)=E_V.
\end{equation*}
This concludes the proof of the proposition.
\end{proof}

%\nocite{*}
\begin{small}
\bibliographystyle{plain}
\bibliography{biblio}
\end{small}
\end{document}